\documentclass[draft,reqno]{amsart}
\usepackage[american]{babel}
\usepackage{amssymb,euscript}
\numberwithin{equation}{section}
\newtheorem{theorem}{Theorem}[section]
\newtheorem{lemma}[theorem]{Lemma}

\theoremstyle{definition}
\newtheorem{definition}[theorem]{Definition}
\theoremstyle{remark}
\newtheorem{remark}[theorem]{Remark}

\def\Re{\operatorname{{Re}}}
\def\Im{\operatorname{{Im}}}
\DeclareMathOperator*{\esssup}{ess\,sup}
\DeclareMathOperator*{\supp}{supp}

\author[Faminskii]{Andrei V. Faminskii}
\address{Peoples' Friendship University of Russia (RUDN University), 6 Miklukho-Maklaya Street, Moscow, 117198, Russian Federation}
\thanks{The work was supported by the Ministry of Science and Higher Education of the Russian Federation: agreement no. 075-03-2020-223/3 (FSSF-2020-0018).}
\email{afaminskii@sci.pfu.edu.ru}

\title[Higher order nonlinear Schr\"{o}dinger equation]{The higher order nonlinear Schr\"{o}dinger equation with quadratic nonlinearity on the real axis}

\date{}

\begin{document}
\maketitle

\begin{abstract}
The initial value problem is considered for a higher order nonlinear Schr\"odinger equation with quadratic nonlinearity. Results on existence and uniqueness of weak solutions are obtained. In the case of an  effective at infinity additional damping large-time decay of solutions without any smallness assumptions is also established. The main difficulty of the study is the non-smooth character of the nonlinearity.
\end{abstract}

\section{Introduction. Notation. Description of main results}\label{S1}

The nonlinear Sch\"odinger equation (NLS) in one spatial dimension
$$
i u_t +a u_{xx} + \lambda |u|^p u =0,
$$
where $u=u(t,x)$ is a complex-valued function, $p>0$, $a, \lambda$ are real non-zero constants,
and the Korteweg--de~Vries equation (KdV)
$$
u_t+bu_x+u_{xxx}+\beta u u_x=0,
$$
where $u=u(t,x)$ is a real-valued function, $b, \beta$ are real constants, $\beta\ne 0$, the modified KdV equation
$$
u_t+bu_x+u_{xxx}+\beta u^2 u_x=0
$$
or more generally the $k$-generalized KdV equation (k-gKdV)
$$
u_t+bu_x+u_{xxx}+\beta u^k u_x=0
$$
are ones of the most famous examples of nonlinear dispersive equations (see, for example, \cite{LP}).

Each of these equations has its own physical meaning, but, moreover, certain common mathematical properties. Consider, for example, the initial value problem with the initial function $u_0(x)$.
Multiply the NLS equation by $2\bar u(t,x)$, take the imaginary part and integrate over $\mathbb R$, then since
$$
\int \Im 2i u_t \bar u \, dx = \int (\bar u u_t + \bar u_t u)\,dx = \frac{d}{dt} \int |u|^2\,dx,
$$
$$
\int \Im 2a u_{xx} \bar u\, dx = \frac{a}{i} \int (u_{xx} \bar u - \bar u_{xx} u)\, dx =0,
$$
$$
\Im 2\lambda |u|^pu \bar u =0,
$$
where here and further we drop the limits in integrals over the whole real line $\mathbb R$, it follows the conservation law
$$
\|u(t,\cdot)\|_{L_2(\mathbb R)} = \|u_0\|_{L_2(\mathbb R)}.
$$
But the same conservation law holds for the KdV equation and its generalizations and it can be obtained by the multiplication of the equation by $2u(t,x)$ and consequent integration over $\mathbb R$.

Moreover, for the NLS and k-gKdV equations the following quantities, which are usually referred to as the energy, are also preserved by the solution flow:
$$
\int \Bigl(|u_x|^2 -\frac{2\lambda}{a(p+2)} |u|^{p+2}\Bigr)\,dx,\quad 
\int \Bigl(u_x^2 -\frac{2}{(k+1)(k+2)} u^{k+2}\Bigr)\,dx.
$$

Next, consider the corresponding linear analogs of the NLS and the KdV equations:
$$
i u_t + a u_{xx} =0,\qquad u_t +bu_x +u_{xxx}=0.
$$
Then the images of the Fourier transform of the solutions to the initial value problem are written as follows:
$$
\widehat u(t,\xi) = \widehat u_0(\xi) e^{ia \xi^2 t},\quad \widehat u(t,\xi) = \widehat u_0(\xi) e^{i(\xi^3 -b\xi)t},
$$ 
that is in the form
$$
\widehat u(t,\xi) = \widehat u_0(\xi) e^{i P(\xi)t},
$$
where $P(\xi)$ is a real polynomial with $P''(\xi)\ne 0$, which shows the dispersive nature of these equations.

Therefore, one can combine these equations and consider the following one:
$$
i u_t + a u_{xx} + ib u_x +i u_{xxx} + \lambda |u|^p u + i\beta (|u|^p u)_x + i\gamma (|u|^p)_x u =0,
$$
where $u=u(t,x)$ is a complex-valued function, $a$, $b$, $\lambda$, $\beta$, $\gamma$ are real constants. This equation is called the higher order nonlinear  Sch\"odinger equation (HNLS)
and also has the physical meaning. It has applications in the propagation of femtosecond optical
pulses in a monomode optical fiber, accounting for additional effects such as third order dispersion, self-steeping of the pulse, and self-frequency shift (see \cite{Fib, HK, Kod, KC}, and the references therein).

Similarly to the NLS and k-gKdV equations it can be shown that the norm in $L_2(\mathbb R)$ of a solution to the initial value problem is also preserved. On the contrary, if one tries to derive the analogue of the conservation law for the energy, the following identity is obtained:
\begin{multline*}
\frac{d}{dt} \int \Bigl[ |u_x|^2  + \frac{i}{\beta+\gamma}\Bigl(\lambda - \frac{a(3\beta+2\gamma)}{3}\Bigr) u \bar u_x 
-\frac{2(3\beta +2\gamma)}{3(p+2)} |u|^{p+2}\Bigr]\, dx \\ -
\frac{\gamma}{3} \int \bigl(|u|^p\bigr)_x \bigl(|u|^2\bigr)_{xx} \,dx =0.
\end{multline*}
Therefore, the integral of the expression in the square brackets is preserved only if either $\gamma=0$ or $p=2$,
since
$$
\int_{\mathbb R} \bigl(|u|^2\bigr)_x \bigl(|u|^2\bigr)_{xx} \,dx =0.
$$

Another serious obstacle in the study of such an equation is the non-smoothness of the function $|u|^p$ except the special cases of even natural values of $p$. Just in the case $p=2$, that is for the equation
$$
i u_t + a u_{xx} + ib u_x +i u_{xxx} + \lambda |u|^2 u + i\beta \bigl(|u|^2 u\bigr)_x + i\gamma (|u|^2)_x u =0,
$$
were previously obtained results on local and global well-posedness of the initial value problem. In particular, in \cite{Laurey} local well-posedness was proved for the initial data $u_0$ in $H^s(\mathbb R)$, $s>3/4$, and global well-posedness in $H^s(\mathbb R)$, $s\geq 1$, if $\beta+\gamma\ne 0$. In \cite{Staf} the local result was improved up to $s\geq 1/4$. In \cite{CL} the same result was obtained for the analogous equation with variable (depending on $t$) coefficients. Under certain relation between the coefficients the local result was extended to the global one in \cite{Car} if $s>1/4$. In \cite {CN} global well-posedness was proved in certain weighted subspaces of $H^2(\mathbb R)$, where power weights were effective at $\pm\infty$ (looking ahead, in the present article the weights are effective only at $+\infty$); the argument used global estimates from \cite{Laurey}. A unique continuation property was obtained in \cite{CP}. 

No results for the initial value problem for the HNLS equation with other values of $p$ have been previously established.

In \cite{CCFSV} an initial-boundary value problem on a bounded interval $I$ for the HNLS equation was considered. In the case $p\in [1,2]$ and the initial function $u_0 \in H^s(I)$, $0\leq s \leq 3$, results on global existence and uniqueness of mild solutions were obtained. For $u_0\in L_2(I)$ the result on global existence was extended either to $p\in (2,3)$ or $p\in (2,4)$, $\gamma=0$. Moreover, after addition to the equation of the damping term $id(x)u$, where the non-negative function $d$ was strictly positive on a certain sub-interval, large-time decay of solutions was established.  Certain preceding results for the truncated version of the HNLS equation ($\beta=\gamma=0$) can be found in \cite{ASV, BOY, BBV, BV,  Chen, CCPV}.

Note that the value $p=2$ in the nonlinearity in the HNLS equation corresponds in the KdV case to the modified KdV equation, while to the KdV equation itself corresponds the value $p=1$, and this is object of the study of the present article. We consider the initial value problem for an equation
\begin{equation}\label{1.1}
i u_t + a u_{xx} + ib u_x +i u_{xxx} + \lambda |u| u + i\beta (|u| u)_x + id(x)u =0
\end{equation}
with the initial data
\begin{equation}\label{1.2}
u(0,x) = u_0(x),\quad x\in\mathbb R.
\end{equation}
Here $u=u(t,x)$ and $u_0$ are complex-valued function, $a$, $b$, $\lambda$, $\beta$ are real constants. In comparison with the general form of the HNLS equation it is assumed that $\gamma=0$, that is the nonlinearity has the divergent form and, in particular, the conservation law for the analog of the energy holds. The main results consist of theorems on global existence and uniqueness of weak solutions. Moreover, the presence of the damping term $id(x)u$, where the non-negative function $d$ is strictly positive at infinity, gives an opportunity to establish large-time decay of solutions. Similar assumptions, which mean that the damping is effective only at infinity, were previously used in \cite{CDFN} for the case of the KdV equation itself. For existence and uniqueness results the presence of the damping term is irrelevant.

The initial function $u_0$ is assumed to be from the following weighted at $+\infty$ $L_2$ space: $(1+x_+)^{3/4} u_0 \in L_2(\mathbb R)$ (here and further $x_+ = \max(x,0)$). Previously the corresponding theory was developed for the KdV equation itself in \cite{KF}. The result on global well-posedness of the initial value problem obtained there was not optimal with respect to the properties of the initial function. Later, for example, in the paper \cite{B} global well-posedness was established for $u_0\in L_2(\mathbb R)$ without any additional weights and even for $u_0\in H^s(\mathbb R)$ for negative values of $s$ (see, for example, \cite{Guo}). However, the methods of that papers required the use of the Bourgain spaces, where a very sophisticated harmonic analysis was applied, in particular, closely related to the concrete nonlinearity $(u^2)_x$. Possibility to apply the technique of the Bourgain spaces to the nonlinearity $(|u|u)_x$ is not evident and is an open problem. On the contrary, the class of well-posedness in \cite{KF} was very simple, did not require any smoothness and consisted of functions $u(t,x)$ such that $(1+x_+)^{3/4} \in L_\infty(0,T;L_2(\mathbb R))$. It turned out that the ideas of the paper \cite{KF} can be applied to equation \eqref{1.1}. Although the methods of the last paper look like rough and archaic in comparison with \cite{B}, they are effective for the considered problem. However, a lot of preliminary work had to be done.

Let $\Pi_T = (0,T)\times \mathbb R$. Let $L_p = L_p(\mathbb R)$, $H^s = H^s(\mathbb R)$, $C_b = C_b(\mathbb R)$ (the subscript $b$ means bounded functions), $C_b^k = C_b^k(\mathbb R)$, $\EuScript S = \EuScript S(\mathbb R)$, $\EuScript S' = \EuScript S'(\mathbb R)$. The notation $C_w([0,T];\dots)$ means a weakly continuous map.

Let $\psi(x) \not\equiv \text{const}$ be a non-negative continuous on $\mathbb R$ function. Define a special weighted space:
$$
L_p^{\psi(x)} = L_p^{\psi(x)}(\mathbb R) = \{\varphi(x): \varphi \psi^{1/2} \in L_p\},
$$
endowed with the natural norm.
Let $L_p^0 = L_p$ and for $\alpha \ne 0$ 
$$
L_p^\alpha = L_p^\alpha (\mathbb R) = L_p^{(1+x_+)^{2\alpha}} = \{\varphi(x): (1+x_+)^{\alpha} \varphi(x) \in L_p\}.
$$

The notion of a weak solution of the considered problem is understood in the following sense.

\begin{definition}\label{D1.1}
Let $T>0$, $u_0\in L_2$, $d\in L_2$. A function $u\in L_\infty(0,T;L_2)$ is called a weak solution to problem \eqref{1.1}, \eqref{1.2} if for any function $\phi \in C^1([0,T];L_2) \cap C([0,T]; H^3)$, $\phi\big|_{t=T} =0$, the following equality holds:
\begin{multline}\label{1.3}
\iint_{\Pi_T} \bigl( i u \phi_t - a u \phi_{xx} + ib u \phi_x  + i u \phi_{xxx} -\lambda |u| u \phi  + i\beta |u| u \phi_x -id(x) u\phi\bigr)\,dxdt  \\+
\int u_0 \phi\big|_{t=0}\,dx =0.
\end{multline}
\end{definition}

Note that under the hypothesis of Definition~\ref{D1.1} $|u|u\in L_\infty(0,T;L_1)$, $\phi\in C([0,T];C^2_b)$ and so the integral in the left-hand side of \eqref{1.3} exists.

\begin{definition}\label{D1.2}
We say that the function $d(x)$ satisfies Condition~A, if it is non-negative on $\mathbb R$ and
there exist positive constants $d_0$ and $R_0$ such that
\begin{equation}\label{1.4}
d(x)\geq d_0\qquad \text{for}\quad |x|\geq R_0.
\end{equation}
\end{definition}

The main result of the paper is the following theorem.

\begin{theorem}\label{T1.1}
Let $u_0\in L_2^{3/4}$, $d\in L_\infty$. Then for any $T>0$ in the strip $\Pi_T$ there exists a unique weak solution $u\in C_w([0,T];L_2^{3/4})$ to problem \eqref{1.1}, \eqref{1.2}. If in addition the function $d(x)$ satisfies Condition~A, then there exists a positive constant $\gamma$ and for any $M>0$ a positive constant $c(M)$ such that if $\|u_0\|_{L_2} \leq M$
\begin{equation}\label{1.5}
\|u(t,\cdot)\|_{L_2}^2 \leq c(M) e^{-\gamma t}\quad \forall\ t\geq 0.
\end{equation}
\end{theorem}

Introduce certain auxiliary notation. For $\alpha\geq 0$ let
$$
H^{1,\alpha} = \{\varphi(x): \varphi, \varphi' \in L_2^\alpha\}
$$
endowed with the natural norm.

Let $\eta(x)$ denotes a cut-off function, namely, $\eta$ is an infinitely smooth non-decreasing function on $\mathbb R$ such that $\eta(x)=0$ for $x\leq 0$, $\eta(x)=1$ for $x\geq 1$, $\eta(x)+\eta(1-x) \equiv 1$.

Define the following weight function $\rho_{\alpha,\epsilon}(x)$ for $\alpha\geq 0$, $\epsilon>0$: let $\rho_{\alpha,\epsilon} \in C^\infty(\mathbb R)$ be an increasing function such that $\rho_{\alpha,\epsilon}(x) = e^{2\epsilon x}$ for $x\leq -1$, $\rho_{\alpha,\epsilon}(x) = (1+x)^{2\alpha}$ if $\alpha>0$ and $\rho_{0,\epsilon} = 2 - \ln^{-1}(x+e)$ for $x\geq 0$, $\rho'_{\alpha,\epsilon}(x) >0$ for $x\in (-1,0)$. Note that if either $\rho(x) \equiv \rho_{\alpha,\epsilon}(x)$ or $\rho(x) \equiv\rho'_{\alpha,\epsilon}(x)$ then for any natural $j$
\begin{equation}\label{1.6}
|\rho^{(j)}(x)| \leq c(j,\alpha,\epsilon)\rho(x)\quad \forall x\in\mathbb R.
\end{equation}

\begin{definition}\label{D1.3}
For any $T>0$ define a class $X(\Pi_T)$ of functions $u(t,x)$ such that
\begin{align*}
u\in C([0,T]; H^1),\quad& u_x \in L_6(0,T; C_b),\\
u_{xx} \in C_b(\mathbb R; L_2(0,T)),\quad& u\in L_2(\mathbb R; C[0,T]).
\end{align*}
\end{definition} 

For $\delta \geq 0$ and $\theta\geq 0$ define a function
\begin{equation}\label{1.7}
g_\delta(\theta) \equiv (\theta +\delta)^{1/2}.
\end{equation}

The paper is organized as follows.  In Section~\ref{S2} an auxiliary linear problem is considered. In Section~\ref{S3} we present the results on existence of solutions to the original problem, and in Section~\ref{S4} --- on uniqueness and continuous dependence on initial data. Section~\ref{S5} contains certain results on internal regularity of weak solutions, which, in particular, are used in Section~\ref{S6} devoted to large-time decay of solutions. Certain technical items related to the properties of the fundamental solution of the corresponding linear operator and differentiation of $|u|$ are presented in Appendixes~\ref{A} and~\ref{B}.

\section{Auxiliary linear problem}\label{S2}

Consider a linear problem
\begin{gather}\label{2.1}
i u_t +a u_{xx} +i b u_x+i u_{xxx}  = f(t,x),\\ 
\label{2.2}
u(0,x) = u_0(x),\quad x\in \mathbb R.
\end{gather}

\begin{lemma}\label{L2.1}
Let $u_0 \in \EuScript S$, $f\in C^\infty([0,T];\EuScript S)$, then there exists a unique solution to problem \eqref{2.1}, \eqref{2.2} $u\in C^\infty([0,T];\EuScript S)$.
\end{lemma}

\begin{proof}
The unique solution to the considered problem in the considered space is constructed via the Fourier transform and is given by a formula
\begin{equation}\label{2.3}
\widehat u(t,\xi) = \widehat u_0(\xi) e^{i(\xi^3 -a \xi^2 -b\xi)t} -i
\int_0^t \widehat f(\tau,\xi) e^{i(\xi^3 -a \xi^2 -b\xi)(t-\tau)}\,d\tau,
\end{equation}
where
$$
\widehat u(t,\xi) \equiv \int e^{-i x\xi} u(t,x)\,dx
$$
with similar notation for $\widehat u_0$ and $\widehat f$.
\end{proof}

\begin{definition}\label{D2.1}
Let $u_0\in \EuScript S'$, $f\equiv f_1 + f_{2x}$, where $f_1,f_2 \in \bigl(C^\infty([0,T];S)\bigr)'$. A function $u\in \bigl(C^\infty([0,T];S)\bigr)'$ is called a generalized solution to problem \eqref{2.1}, \eqref{2.2} in a strip $\Pi_T$ if for any function $\phi \in C^\infty([0,T];S)$, $\phi(T,x) \equiv 0$,  it verifies an equality
\begin{equation}\label{2.4}
\langle u, i\phi_t - a\phi_{xx} + i b \phi_x+i \phi_{xxx} \rangle +\langle f_1,\phi\rangle - \langle f_2, \phi_x\rangle+ \langle u_0, \phi\big|_{t=0}\rangle  =0. 
\end{equation}
\end{definition}

\begin{remark}\label{R2.1}
Any weak solution to problem \eqref{2.1}, \eqref{2.2} in the sense of Definition \ref{D1.1} (with corresponding changes) is, of course, a generalized solution. 
\end{remark}

\begin{lemma}\label{L2.2}
The generalized solution to problem \eqref{2.1}, \eqref{2.2} is unique.
\end{lemma}

\begin{proof}
This lemma succeeds from Lemma~\ref{L2.1}  by the standard H\"olmgren argument.
\end{proof}

Now consider the differential operator 
\begin{equation}\label{2.5}
\EuScript L(\partial_t, \partial_x) = \partial_t +\partial_x^3 - i a \partial_x^2 + b\partial_x.
\end{equation}
The general theory of linear differential operators provides that its fundamental solution is written in a form
\begin{equation}\label{2.6}
G(t,x) = \theta(t) \EuScript F^{-1} \bigl[e^{it (\xi^3 - a \xi^2 -b\xi)}\bigr](x),
\end{equation}
where $\theta$ is the Heaviside function.

\begin{theorem}\label{T2.1} 
Let $u_0 \in L_2^{1/4+\varepsilon}$, $f_{11}\in L_1(0,T; L_2^{1/4+\varepsilon})$, $f_{12} \in L_1(0,T;L_1^{-1/4})$, $f_2\in L_1(0,T;L_1^{1/4})$ for certain $\varepsilon>0$ and $T>0$, $f_1 \equiv f_{11} + f_{12}$, $f\equiv f_1 + f_{2x}$. Then there exists a generalized solution $u(t,x)$ to problem \eqref{2.1}, \eqref{2.2} in the strip $\Pi_T$ such that at any point $(t,x)\in \Pi_T$
\begin{multline}\label{2.7}
u(t,x) = \int G(t,x-y) u_0(y) \,dy -i \int_0^t \!\! \int G(t-\tau,x-y) f_1(\tau,y)\, dy d\tau \\ -i
\int_0^t\!\! \int G_x(t-\tau, x-y) f_2(\tau,y)\, dy d\tau.
\end{multline}
\end{theorem}

\begin{proof}
First let $u_0 \in \EuScript S$, $f_1,f_2 \in C^\infty([0,T];\EuScript S)$. Then Lemma~\ref{L2.1} yields that a solution $u\in  C^\infty([0,T];\EuScript S)$ to the considered problem exists. According to \eqref{2.3}
\begin{equation}\label{2.8}
u(t,x) = \EuScript F^{-1}\Bigl[e^{i(\xi^3 -a \xi^2 -b\xi)t} \widehat u_0(\xi) \Bigr](x)   -i
\int_0^t  \EuScript F^{-1} \Bigl[e^{i(\xi^3 -a \xi^2 -b\xi)(t-\tau)} \widehat f(\tau,\xi)\Bigr](x) \,d\tau
\end{equation}
and with the use of \eqref{2.6} we obtain \eqref{2.7} in the smooth case.

Inequalities \eqref{A.26} yield that
\begin{multline*}
\Bigl| \int G(t,x-y) u_0(y) \,dy \Bigr| \leq \frac{c(T)}{t^{1/3}} \Bigl[ \int_x^{+\infty} (1+y-x)^{-1/4} |u_0(y)| \,dy  \\ +
\int_{-\infty}^x e^{-c_0(x-y)^{3/2}T^{-1/2}} |u_0(y)| \,dy \Bigr]
\end{multline*} 
whence, since $(1+y-x)^{-1/4} = (1+y-x)^{-1/2-\varepsilon} (1+y-x)^{1/4+\varepsilon}$ and $(1+y-x) \leq (1+y_+)(1+x_-)$, it follows that
\begin{multline}\label{2.9}
\Bigl| \int G(t,x-y) u_0(y) \,dy \Bigr| \leq \frac{c(T,\varepsilon)}{t^{1/3}} \Bigl[\Bigl(\int_x^{+\infty} (1+y-x)^{1/2+2\varepsilon} |u_0(y)|^2\, dy \Bigr)^{1/2} \\ +
\Bigl(\int_{-\infty}^x |u_0(y)|^2\, dy\Bigr)^{1/2}\Bigr]  \leq
\frac{c(T,\varepsilon)}{t^{1/3}} (1+x_-)^{1/4+\varepsilon} \|u_0\|_{L_2^{1/4+\varepsilon}},
\end{multline}
where $x_- = \max(-x,0)$.
Similarly,
\begin{multline}\label{2.10} \displaystyle
\int_0^T \Bigl| \int_0^t \!\! \int G(t-\tau,x-y) f_{11}(\tau,y)\, dy d\tau \Bigr| \,dt \\ \leq
c(T,\varepsilon)  (1+x_-)^{1/4+\varepsilon} \int_0^T\!\! \int_0^t \frac{\|f_{11}(\tau,\cdot)\|_{L_2^{1/4+\varepsilon}}}{(t-\tau)^{1/3}} \,d\tau dt \\ \leq
c_1(T,\varepsilon)  (1+x_-)^{1/4+\varepsilon} \|f_{11}\|_{L_1(0,T;L_2^{1/4+\varepsilon})}.
\end{multline}
Next, since $1+y_+ \leq (1+x_+)(1+y-x)$ if $x\leq y$, it follows that
\begin{multline}\label{2.11}
\int_0^T \Bigl| \int_0^t \!\! \int G(t-\tau,x-y) f_{12}(\tau,y)\, dy d\tau \Bigr| \,dt \\ \leq
c(T) \int_0^T\!\! \int_0^t \frac1{(t-\tau)^{1/3}} \Bigl[ \int_x^{+\infty} (1+y-x)^{-1/4} |f_{12}(\tau,y)|\, dy \\+
\int_{-\infty}^x e^{-c_0(x-y)^{3/2}T^{-1/2}} |f_{12}(\tau,y)|\, dy \Bigr]\,d\tau dt  \leq
c_1(T) (1+x_+)^{1/4} \|f_{12}\|_{L_1(0,T;L_1^{-1/4})}.
\end{multline}
Finally,
\begin{multline}\label{2.12}
\int_0^T \Bigl| \int_0^t \!\! \int G_x(t-\tau,x-y) f_2(\tau,y)\, dy d\tau \Bigr| \,dt \\ \leq
c(T) \int_0^T\!\! \int_0^t \frac1{(t-\tau)^{3/4}} \Bigl[ \int_x^{+\infty} (1+y-x)^{1/4} |f_2(\tau,y)|\, dy \\+
\int_{-\infty}^x e^{-c_0(x-y)^{3/2}T^{-1/2}} |f_2(\tau,y)|\, dy \Bigr]\,d\tau dt  \leq
c_1(T) (1+x_-)^{1/4} \|f_2\|_{L_1(0,T;L_1^{1/4})}.
\end{multline}
Estimates \eqref{2.9}--\eqref{2.12} provide an inequality
\begin{multline}\label{2.13}
\|u\|_{L_1(0,T;L_\infty^{(1+|x|)^{-1/2-2\varepsilon}})} \leq c(T,\varepsilon) \bigl[ \|u_0\|_{L_2^{1/4+\varepsilon}} +
\|f_{11}\|_{L_1(0,T;L_2^{1/4+\varepsilon})} \\+\|f_{12}\|_{L_1(0,T;L_1^{-1/4})}  +\|f_2\|_{L_1(0,T;L_1^{1/4})} \bigr].
\end{multline}

Now consider the general case. Approximate the functions $u_0$, $f_{11}$, $f_{12}$ and $f_2$ in the corresponding spaces by the smooth ones. Then on the basis of estimate \eqref{2.13} the function $u(t,x)$ defined by formula \eqref{2.7} can be obtained by the passage to the limit in the space $L_1(0,T;L_\infty^{(1+|x|)^{-1/2-2\varepsilon}}) \subset \bigl(C^\infty([0,T];S)\bigr)'$ of the sequence of smooth solutions. Moreover, writing down corresponding equalities \eqref{2.4} for smooth solutions and passing to the limit we obtain that the function $u$ also verifies equality \eqref{2.4}. Note, that estimates \eqref{2.9}--\eqref{2.12} also hold in the non-smooth case.
\end{proof}

\begin{lemma}\label{L2.3}
Let $u_0\in L_2$, $f\in L_1(0,T;L_2)$ for certain $T>0$. Then there exists a weak solution to problem \eqref{2.1}, \eqref{2.2} $u\in C([0,T];L_2)$,
\begin{equation}\label{2.14}
\|u\|_{C([0,T];L_2)} \leq \|u_0\|_{L_2} +\|f\|_{L_1(0,T;L_2)}
\end{equation}
and for a.e. $t\in (0,T)$
\begin{equation}\label{2.15}
\frac{d}{dt} \int |u|^2\,dx = 2 \Im \int f \bar u\,dx.
\end{equation}
Moreover, for any non-negative non-decreasing function $\psi\in C^3_b$, $\psi' \not\equiv 0$, 
$u_x \in L_2(0,T; L_2^{\psi'(x)})$,
\begin{equation}\label{2.16}
\|u_x\|_{L_2(0,T; L_2^{\psi'(x)})} \leq c(T,\psi)\bigl(\|u_0\|_{L_2} +\|f\|_{L_1(0,T;L_2)}\bigr)
\end{equation}
and for a.e. $t\in (0,T)$
\begin{multline}\label{2.17}
\frac{d}{dt}\int |u(t,x)|^2\psi(x)\, dx +3 \int |u_x|^2\psi' \,dx = 2a \Im \int u_x \bar u \psi'\, dx \\ +
 \int |u|^2 \psi'''\, dx + b \int |u|^2 \psi' \, dx + 2\Im \int f \bar u \psi\, dx. 
\end{multline}
\end{lemma}

\begin{proof}
First assume that the functions $u_0$ and $f$ satisfy the hypothesis of Lemma~\ref{L2.1} and consider the corresponding smooth solution $u(t,x)$. Let $\psi\in C^3_b$. Multiply equation \eqref{2.1} by $2\bar u(t,x) \psi(x)$, integrate over $\mathbb R$ and take the imaginary part, then
\begin{multline}\label{2.18}
\int \Im 2iu_t \bar u \psi\,dx + a \int 2\Im u_{xx} \bar u\psi \, dx + b \int \Im 2i u_x \bar u\psi \, dx   + \int \Im 2i u_{xxx} \bar u\psi \, dx  \\= 2\Im \int f \bar u \psi\, dx. 
\end{multline}
Here
$$
\int \Im 2i u_t \bar u\psi \,dx = \int \Re 2 u_t \bar u\psi \, dx = \int (u_t \bar u + u \bar u_t)\psi \,dx = \frac{d}{dt} \int |u|^2\psi \,dx,
$$
$$
\int 2\Im u_{xx} \bar u\psi \,dx = -2\int \Im (|u_x|^2\psi +u_x \bar u \psi')\, dx = -2\Im \int u_x \bar u\psi'\, dx,
$$
$$
\int \Im 2i u_{x}\bar u\psi \,dx = \int 2\Re u_{x} \bar u \psi \, dx = \int (u_{x}\bar u + u \bar u_{x})\psi \, dx  = 
-\int |u|^2\psi' \,dx,
$$
\begin{multline*}
\int \Im 2i u_{xxx}\bar u\psi\,dx = \int 2\Re u_{xxx} \bar u\psi\, dx = \int (u_{xxx}\bar u + u \bar u_{xxx})\psi\, dx \\ = 
-\int (u_{xx} \bar u_x + u_x \bar u_{xx})\psi\, dx - \int (u_{xx}\bar u + u\bar u_{xx})\psi' \,dx \\ =
-\int (|u_x|^2)_x\psi \,dx + 2\int |u_x|^2 \psi'\, dx + \int (|u|^2)_x \psi''\, dx \\ =
3\int |u_x|^2 \psi'\, dx - \int |u|^2 \psi''' \,dx.
\end{multline*}

Choose in \eqref{2.18} $\psi(x) \equiv 1$, then equality \eqref{2.13} follows for smooth solutions which, in turn, 
provides  estimate \eqref{2.14} and for any $t\in [0,T]$ an equality 
\begin{equation}\label{2.19}
\int |u(t,x)|^2\,dx = \int |u_0|^2\,dx + 2 \Im \int_0^t \int f(\tau,x) \bar u(\tau,x)\, dx d\tau.
\end{equation}
Applying closure on the basis of estimate \eqref{2.14} for smooth solutions we obtain the result on existence in the space $C([0,T];L_2)$, estimate \eqref{2.14} for such solutions and equality \eqref{2.19}. Since $\int f(t,x) \bar u(t,x)\, dx \in L_1(0,T)$, the right-hand side of equality \eqref{2.19} is absolutely continuous with respect to $t$ and equality \eqref{2.15} succeeds.

Now let $\psi$ be non-negative, non-decreasing and $\psi\not\equiv 0$. Then equality \eqref{2.18} yields equality \eqref{2.17} for smooth solutions. With the use of estimate \eqref{2.14} this equality insures estimate \eqref{2.16}. Closure finishes the proof.
\end{proof}

\begin{lemma}\label{L2.4}
Let $u\in L_\infty(0,T;L_2^{\alpha})$, $\alpha\geq 3/4$,  be a weak solution to problem \eqref{2.1}, \eqref{2.2} in a certain strip $\Pi_T$, where $u_0\equiv 0$, $f\equiv f_1+ f_{2x}$, $f_1 \in L_{p_1}(0,T;L_1^{\alpha+1/4})$, $f_2\in L_{p_2}(0,T;L_1^{\alpha+3/4})$ for certain $p_1> 3/2$, $p_2> 4$. Then for any $\epsilon>0$
\begin{multline}\label{2.20}
\|u\|_{L_\infty(0,T;L_2^{\rho_{\alpha,\epsilon}(x)})} \leq c(T,\alpha,\epsilon,p_1,p_2) \bigl[ \|f_1\|_{L_{p_1}(0,T;L_1^{\rho_{\alpha,\epsilon}(x)(1+x_+)^{1/2}})} \\+
 \|f_2\|_{L_{p_2}(0,T;L_1^{\rho_{\alpha,\epsilon}(x)(1+x_+)^{3/2}})}.
\end{multline}
\end{lemma}

\begin{proof}
First of all note that $u\in L_\infty(0,T;L_2^{\rho_{\alpha,\epsilon(x)}})$. Apply Theorem~\ref{T2.1} and Lemma~\ref{LA.3}, then equality \eqref{2.7} and inequality \eqref{A.30} (for $f_1$ in the case $n=0$ and for $f_2$ in the case $n=1$) imply that for any $t\in (0,T]$
\begin{multline*}
\|u(t,\cdot)\|_{L_2^{\rho_{\alpha,\epsilon}(x)}} = \|u(t,\cdot)\rho^{1/2}_{\alpha,\epsilon}\|_{L_2} \leq 
\int_0^t \!\! \int\|G(t-\tau,\cdot-y)\rho^{1/2}_{\alpha,\epsilon}\|_{L_2} |f_1(\tau,y)|\, dy d\tau \\ +
\int_0^t \!\! \int\|G_x(t-\tau,\cdot-y)\rho^{1/2}_{\alpha,\epsilon}\|_{L_2} |f_2(\tau,y)|\, dy d\tau \\ \leq
c\int_0^t \frac1{(t-\tau)^{1/3}}\int |f_1(\tau,y)| \rho_{\alpha,\epsilon}^{1/2}(y)(1+y_+)^{1/4}\, dy d\tau \\+
c\int_0^t \frac1{(t-\tau)^{3/4}}\int |f_2(\tau,y)| \rho_{\alpha,\epsilon}^{1/2}(y)(1+y_+)^{3/4}\, dy d\tau  \\ \leq
c_1\|f_1\|_{L_{p_1}(0,T;L_1^{\rho_{\alpha,\epsilon}(x)(1+x_+)^{1/2}})} +
c_1\|f_2\|_{L_{p_2}(0,T;L_1^{\rho_{\alpha,\epsilon}(x)(1+x_+)^{3/2}})}.
\end{multline*}
\end{proof}

\begin{lemma}\label{L2.5}
Let $u_0 \in L_2^\alpha$, $f\in L_1(0,T;L_2^\alpha)$ for certain $\alpha>0$, $T>0$. Then there exists a weak solution $u(t,x)$ to problem \eqref{2.1}, \eqref{2.2} such that $u\in C([0,T];L_2^\alpha)$, $u_x \in L_2(0,T; L_2^{\rho'_{\alpha,\epsilon}(x)})$ for any $\epsilon>0$,
\begin{equation}\label{2.21}
\|u\|_{C([0,T];L_2^\alpha)} + \|u_x\|_{L_2(0,T; L_2^{\rho'_{\alpha,\epsilon}(x)})} \leq c(T,\alpha,\epsilon) \bigl[ \|u_0\|_{L_2^\alpha} + \|f\|_{L_1(0,T;L_2^\alpha)} \bigr]
\end{equation}
and for a.e $t\in (0,T)$ and $\rho(x) \equiv \rho_{\alpha,\epsilon}(x)$
\begin{multline}\label{2.22}
\frac{d}{dt}\int |u(t,x)|^2\rho(x)\, dx +3\int |u_x|^2\rho' \,dx = 2a \Im \int u_x \bar u \rho'\, dx \\ +
\int |u|^2 \rho'''\, dx +  b\int |u|^2 \rho' \, dx + 2\Im \int f \bar u \rho\, dx. 
\end{multline}
\end{lemma}

\begin{proof}
First assume that the functions $u_0$ and $f$ satisfy the hypothesis of Lemma~\ref{L2.1} and consider the corresponding smooth solution $u(t,x)$. Multiply equation \eqref{2.1} by $2\bar u(t,x)\rho(x)$, integrate over $\mathbb R$ and take the imaginary part, then equality \eqref{2.18} follows, where $\psi$ is substituted by $\rho$, whence equality \eqref{2.22} follows. In turn, with the use of properties \eqref{1.6} this equality provides estimate \eqref{2.21}. The end of the proof is similar to Lemma~\ref{L2.3}.
\end{proof}

\begin{lemma}\label{L2.6}
Let $u_0\in H^1$, $f\in L_1(0,T;H^1)$ for certain $T>0$. Then there exists a unique solution to problem \eqref{2.1}, \eqref{2.2} $u\in X(\Pi_T)$,
\begin{equation}\label{2.23}
\|u\|_{X(\Pi_T)} \leq c(T) \left( \|u_0\|_{H^1} + \|f\|_{L_1(0,T;H^1)} \right)
\end{equation}
and for a.e $t\in (0,T)$
\begin{equation}\label{2.24}
\frac{d}{dt} \int |u_x|^2\,dx = 2 \Im \int f_x \bar u_x\,dx,
\end{equation}
\begin{multline}\label{2.25}
\frac23 \frac{d}{dt} \int |u|^3\,dx + 2a \Im \int |u|_x u \bar u_x \,dx + 3\int |u|_x |u_x|^2\, dx  -
\int |u|_x \bigl(|u|^2\bigr)_{xx}\, dx \\= 2\Im \int f |u| \bar u\, dx,
\end{multline}
\begin{equation}\label{2.26}
i\frac{d}{dt} \int u \bar u_x\, dx = 2 \Re \int f \bar u_x\, dx.
\end{equation}
Moreover, for any non-negative non-decreasing function $\psi\in C^3_b$, $\psi' \not\equiv 0$, 
$u_{xx} \in L_2(0,T; L_2^{\psi'(x)})$,
\begin{equation}\label{2.27}
\|u_{xx}\|_{L_2(0,T; L_2^{\psi'(x)})} \leq c(T,\psi)\bigl(\|u_0\|_{H^1} +\|f\|_{L_1(0,T;H^1)}\bigr)
\end{equation}
and for a.e. $t\in (0,T)$
\begin{multline}\label{2.28}
\frac{d}{dt}\int |u_x(t,x)|^2\psi(x)\, dx +3 \int |u_{xx}|^2\psi' \,dx = 2a \Im \int u_{xx} \bar u_x \psi'\, dx \\ +
 \int |u_x|^2 \psi'''\, dx + b \int |u_x|^2 \psi' \, dx + 2\Im \int f_x \bar u_x \psi\, dx,
\end{multline}
\begin{multline}\label{2.29}
\frac23 \frac{d}{dt} \int |u|^3\psi \,dx + 2a \Im \int \bigl(|u|\psi\bigr)_x u \bar u_x \,dx + 
3\int \bigl(|u|\psi\bigr)_x |u_x|^2\, dx \\ -
\int \bigl(|u|\psi\bigr)_x \bigl(|u|^2\bigr)_{xx}\, dx -\frac23 b \int |u|^3\psi' \,dx 
= 2\Im \int f |u| \bar u \psi\, dx.
\end{multline}
\end{lemma}

\begin{proof}
Similarly to \eqref{2.14}
$$
\|u_x\|_{C([0,T];L_2)} \leq \|u_0'\|_{L_2} + \|f_x\|_{L_1(0,T;L_2)}.
$$
Next, apply estimates from \cite{CT}: if $f\equiv 0$
\begin{align*}
\|u\|_{L_6(0,T; L_\infty)} &\leq c(T) \|u_0\|_{L_2},\\
\|u_x\|_{L_\infty(\mathbb R;L_2(0,T))} &\leq c(T)\|u_0\|_{L_2},\\
\|u\|_{L_2(\mathbb R;L_\infty(0,T))} &\leq c(T)\|u_0\|_{H^1}.
\end{align*}
Note that in the smooth case all $L_\infty$ can be changed to $C_b$ norms and then by closure the same can be done in the general case. After that  by the standard argument we consider the non-homogeneous case and obtain the result on existence in the space $X(\Pi_T)$ with estimate \eqref{2.23}.

In order to prove other properties as in the previous proof first consider smooth solutions. Let $\psi\in C^3_b$. Multiply \eqref{2.1} by $-2\bigl( \bar u_{x}\psi(x)\bigr)_x$, integrate over $\mathbb R$ and take the imaginary part, then
\begin{multline}\label{2.30}
-\int \Im 2i u_t \bigl(\bar u_{x}\psi\bigr)_x\,dx - a \int 2\Im u_{xx} \bigl(\bar u_{x}\psi\bigr)_x\, dx - b \int \Im 2i u_x \bigl(\bar u_{x}\psi\bigr)_x\, dx  \\ -\int \Im 2i u_{xxx}  \bigl(\bar u_{x}\psi\bigr)_x\, dx   = -2\Im \int f \bigl(\bar u_{x}\psi\bigr)_x\, dx. 
\end{multline}
Here
\begin{multline*}
-\int \Im 2i u_t \bigl(\bar u_{x}\psi\bigr)_x\,dx = \int \Re 2 u_{tx} \bar u_x\psi\, dx \\= 
\int (u_{tx} \bar u_x + u_x \bar u_{tx})\psi \,dx = \frac{d}{dt} \int |u_x|^2 \psi\,dx,
\end{multline*}
\begin{multline*}
-\int 2\Im u_{xx} \bigl(\bar u_{x}\psi\bigr)_x\,dx = -\int 2\Im |u_{xx}|^2 \psi\, dx - 2\Im \int u_{xx} \bar u_x \psi' \,dx  \\ =
- 2\Im \int u_{xx} \bar u_x \psi' \,dx,
\end{multline*}
\begin{multline*}
-\int \Im 2i u_{x}\bigl(\bar u_{x}\psi\bigr)_x\,dx = -\int 2\Re u_{x} \bar u_{xx}\psi \, dx  - 2\int |u_x|^2 \psi' \,dx  \\ = 
-\int \bigl(|u_x|^2\bigr)_x \psi \,dx - 2\int |u_x|^2 \psi' \,dx = - \int |u_x|^2 \psi' \,dx,
\end{multline*}
\begin{multline*}
-\int \Im 2i u_{xxx}\bigl(\bar u_{x}\psi\bigr)_x\,dx = -\int 2\Re u_{xxx} (\bar u_{xx}\psi + \bar u_x \psi')\, dx  \\ = 
-\int (u_{xxx}\bar u_{xx} + \bar u_{xxx} u_{xx})\psi \, dx  +2\int |u_{xx}|^2\psi'\, dx + 2\int \Re u_{xx} \bar u_x \psi'' \,dx \\ =
3 \int |u_{xx}|^2 \psi' \,dx  - \int |u_x|^2 \psi''' \,dx.
\end{multline*}
Choosing either $\psi(x)\equiv 1$ or $\psi\not\equiv 0$ non-negative and non-decreasing, we derive from \eqref{2.30} equalities \eqref{2.24} and \eqref{2.28} for smooth solutions. Then estimate \eqref{2.23} and equality \eqref{2.28} provide estimate \eqref{2.27}.

Next, multiply \eqref{2.1} by $2|u(t,x)|\bar u(t,x) \psi(x)$, integrate over $\mathbb R$ and take the imaginary part, then
\begin{multline}\label{2.31}
\int \Im 2iu_t |u|\bar u \psi\,dx + a \int 2\Im u_{xx} |u|\bar u\psi\, dx + b \int \Im 2i u_x |u|\bar u\psi\, dx \\ +\int \Im 2i u_{xxx} |u|\bar u\psi\, dx  = 2\Im \int f |u|\bar u\psi\, dx. 
\end{multline}
Here
\begin{multline*}
\int \Im 2i u_t |u|\bar u \psi \,dx = \int \Re 2 |u| u_t \bar u \psi \, dx  = \int |u| ( u_t \bar u + \bar u_{t} u )\psi \,dx \\ =\int \bigl(|u|^2\bigr)^{1/2} \bigl( |u|^2\bigr)_t \psi\,dx 
= \frac23 \int \Bigl(\bigl( |u|^2\bigr)^{3/2}\Bigr)_t \psi\,dx = \frac23 \frac{d}{dt} \int |u|^3\psi\,dx,
\end{multline*}
\begin{multline*}
\int 2\Im u_{xx} |u| \bar u \psi\,dx = -\int 2\Im \bigl(|u|\psi\bigr)_x \bar u u_x\,dx -\int 2\Im |u| u_x \bar u_x \psi \,dx \\=  
2\Im \int \bigl(|u|\psi\bigr)_x u \bar u_x \,dx,
\end{multline*}
\begin{multline*}
\int \Im 2i u_{x} |u| \bar u\psi\,dx = \int \Re |u| (u_x \bar u + \bar u_x u)\psi\,dx = \int (|u|^2)^{1/2} \bigl(|u|^2\bigr)_x\psi\,dx \\ = -\frac23 \int |u|^3 \psi'\, dx,
\end{multline*}
\begin{multline*}
\int \Im 2i u_{xxx} |u| \bar u \psi\,dx = \int \Re 2 u_{xxx} |u| \bar u \psi\, dx \\ = -\int \Re 2 \bigl(|u|\psi\bigr)_x \bar u u_{xx}\,dx  - \int \Re 2 |u| \bar u_x u_{xx}\psi\, dx  \\=
-\int  \bigl(|u|\psi\bigr)_x (\bar u u_{xx} + \bar u_{xx}  u)\, dx   -\int |u|\psi (\bar u_x u_{xx} +u_x \bar u_{xx})\, dx \\=
-\int \bigl(|u|\psi\bigr)_{x}  \bigl[\bigl(|u|^2\bigr)_{xx} -2|u_x|^2\bigr]\,dx  - \int |u|\psi \bigl(|u_x|^2\bigr)_x\,dx \\ =
3 \int \bigl(|u|\psi\bigr)_x |u_x|^2\,dx - \int \bigl(|u|\psi\bigr)_{x} \bigl(|u|^2\bigr)_{xx}\,dx,
\end{multline*}
and similarly to \eqref{2.30} equality \eqref{2.31} ensures equalities \eqref{2.25} and \eqref{2.29} for smooth solutions. 

Finally, multiply \eqref{2.1} by $2\bar u_x(t,x)$, integrate over $\mathbb R$ and take the real part, then
\begin{multline}\label{2.32}
\int \Re 2iu_t \bar u_{x}\,dx + a \int 2\Re u_{xx} \bar u_{x}\, dx + b \int \Re 2i u_x \bar u_{x}\, dx \\+\int \Re 2i u_{xxx} \bar u_{x}\, dx  = 2\Re \int f \bar u_{x}\, dx. 
\end{multline}
Here
\begin{multline*}
\int \Re 2i u_t \bar u_{x}\,dx = -\int \Im 2 u_t \bar u_{x}\, dx  = i\int (u_{t} \bar u_x - u_x \bar u_{t})\,dx \\ = i\int (u_t \bar u_x + u \bar u_{tx})\,dx= i\frac{d}{dt} \int u \bar u_x\,dx,
\end{multline*}
$$
\int 2\Re u_{xx} \bar u_{x}\,dx = \int (u_{xx}\bar u_x + \bar u_{xx} u_x)\,dx =  0,
$$
$$
\int \Re 2i u_{x}\bar u_{x}\,dx = 0,
$$
$$
\int \Re 2i u_{xxx}\bar u_{x}\,dx = i\int (u_{xxx} \bar u_{x} - \bar u_{xxx} u_x)\, dx = 0,
$$
and, therefore, \eqref{2.32} ensures equality \eqref{2.26} for smooth solutions.

The general case is obtained via closure on the basis of estimates \eqref{2.23} and \eqref{2.27}.
\end{proof}

\section{Existence}\label{S3}

Alongside with equation \eqref{1.1} consider an equation with the regularized nonlinearity
\begin{equation}\label{3.1}
i u_t +a u_{xx} + i b u_x +i u_{xxx} +\lambda g_\delta(|u|^2) u +i\beta \bigl(g_\delta(|u|^2) u\bigr)_x + i d(x)u=0,
\end{equation}
where $\delta \in (0,1]$ and the function $g_\delta$ is defined in \eqref{1.7}. The notion of a weak solution is similar to Definition \ref{D1.1}.

Note that $g_\delta \in C^\infty(\mathbb R_+)$ ($C^\infty(\overline{\mathbb R}_+)$ if $\delta>0$) and for $\theta>0$ uniformly with respect to $\delta\in [0,1]$
\begin{equation}\label{3.2}
|g_\delta(\theta)| \leq \bigl(\theta^{1/2} +1\bigr), \quad 
|g_\delta'(\theta)| \leq \frac12\theta^{-1/2}, \quad 
\quad |g''_\delta(\theta)|  \leq  \frac14\theta^{-3/2}.
\end{equation}

\begin{lemma}\label{L3.1}
Let $d\in H^1$, $T>0$, $u \in X(\Pi_T)$, $\delta \in [0,1]$, define
\begin{equation}\label{3.3}
F(u(t,x)) \equiv -\lambda g_\delta (|u|^2)u -i \beta\bigl(g_\delta(|u|^2) u\bigr)_x -i d(x) u.
\end{equation}
Then $F\in L_2(0,T;H^1)$ and there exists a positive constant $c=c(T,\|d\|_{H^1}, |\lambda|, |\beta|)$ (non-decreasing with respect to its arguments) such that uniformly with respect to $\delta$
\begin{equation}\label{3.4}
\|F\|_{L_2(0,T;H^1)} \leq c\|u\|_{X(\Pi_T)} \bigl( 1 + \|u\|_{X(\Pi_T)}\bigr).
\end{equation}
\end{lemma}

\begin{proof}
From \eqref{3.2} it follows that for $\delta\in [0,1]$
\begin{equation}\label{3.5}
|g_\delta(|u|^2)u|,  |g_\delta(|u|^2) u_x|, |g_\delta(|u|^2) u_{xx}| \leq (1 +|u|) (|u| + |u_x| + |u_{xx}|).
\end{equation}
Next, from \eqref{B.7} it follows that for $\delta \in [0,1]$
\begin{gather}\label{3.6}
\bigl| \bigl(g_\delta(|u|^2)\bigr)_x u\bigr| \leq c |u u_x|,\\
\label{3.7}
\bigl|\bigl(g_\delta(|u|^2)\bigr)_x u_x\bigl|, \bigl(\bigl(g_\delta(|u|^2)\bigr)_x u \bigr)_x\bigr| 
\leq c \bigl(|u_x|^2 + |u u_{xx}|\bigr).
\end{gather}
Here
\begin{equation}\label{3.8}
\iint_{\Pi_T} |u|^{2} |u_{xx}|^2 \,dxdt \leq \int \sup\limits_{t\in (0,T)} |u|^2\,dx \cdot \sup\limits_{x\in \mathbb R} \int_0^T |u_{xx}|^2\,dt  \leq c\|u\|^{4}_{X(\Pi_T)},
\end{equation}
\begin{equation}\label{3.9}
\iint_{\Pi_T} |u_{x}|^4 \,dxdt \leq \sup\limits_{t\in (0,T)} \int |u_x|^2\,dx \cdot \int_0^T \sup\limits_{x\in \mathbb R} |u_{x}|^2\,dt  \leq c(T)\|u\|^{4}_{X(\Pi_T)}.
\end{equation}
Other terms are estimated in a similar and more obvious way.
\end{proof}

\begin{lemma}\label{L3.2}
Let $u_0\in H^1$, $d\in H^1$, $\delta\in (0,1]$. Then there exists $T_0>0$, depending on $\|u_0\|_{H^1}$ and uniform with respect to $\delta$, such that problem \eqref{3.1}, \eqref{1.2} has a unique solution $u_\delta \in X(\Pi_{T_0})$, moreover, uniformly with respect to $\delta$ for certain $R>0$, depending on $\|u_0\|_{H^1}$,
\begin{equation}\label{3.10}
\|u_\delta\|_{X(\Pi_{T_0})} \leq R.
\end{equation}
\end{lemma}

\begin{proof}
Without loss of generality we assume that $T_0 \leq 1$. Let $\|u_0\|_{H^1} \leq M$.

For any $\delta \in (0,1]$ the functions $g_\delta$ verify inequalities \eqref{3.2} uniformly with respect to $\delta$. Therefore, for any $T\in (0,1]$ and $u\in X(\Pi_T)$ estimate \eqref{3.4} is verified uniformly with respect to $T$ and $\delta$. 

Let $R>0$. Define on a set $\widetilde X_R(\Pi_{T}) = \{v\in X(\Pi_{T}): \|v\|_{X(\Pi_{T})} \leq R, v\big|_{t=0} = u_0\}$ a map $u= \Lambda_0 v$, where $u\in X(\Pi_{T})$ is a solution to a linear problem
\begin{equation}\label{3.11}
i u_t +a u_{xx} +ib u_x +i u_{xxx} =F(v(t,x)),\quad u\big|_{t=0} = u_0,
\end{equation}
where the function $F$ is given by formula \eqref{3.3} and the function $u$ is substituted by a function $v\in \widetilde X_R(\Pi_T)$. 

By virtue of Lemma~\ref{L3.1} $F(v)\in L_2(0,T;H^1)$ and uniformly with respect to $\delta$
\begin{equation}\label{3.12}
\|F(v)\|_{L_2(0,T;H^1)} \leq c R (1+R).
\end{equation}
Then by virtue of Lemma~\ref{L2.6} the map $\Lambda_0$ exists and according to \eqref{2.23}
\begin{equation}\label{3.13}
\|u\|_{X(\Pi_{T})} \leq c_0 M + c_0 T^{1/2} R (1+R),
\end{equation}
for certain constant $c_0$ uniform with respect to $\delta$.
Choose
\begin{equation}\label{3.14}
R \geq 2c_0 M,\quad T_0 = \min\left(1, \frac{1}{16c_0^2 (1+2R)^2}\right).
\end{equation}
Then $\Lambda_0$ maps $\widetilde X_R(\Pi_T)$ into itself  for any $T\leq T_0$.

Next we show that for any $\delta\in (0,1]$ there exists $T(\delta)\in (0,T_0]$ such that for any $T\in (0,T(\delta)]$ on the set $\widetilde X_R(\Pi_{T})$ the map $\Lambda_0$ is a contraction. To this end one has to estimate $F(v) -F(\widetilde v)$ and $\bigl(F(v) -F(\widetilde v)\bigr)_x$. In particular, consider $\bigl(g_\delta(|v|^2)v\bigr)_{xx} - \bigl(g_\delta(|\widetilde v|^2)\widetilde v\bigr)_{xx}$. Note that
$$
\bigl(g_\delta(|v|^2)v\bigr)_{xx} = g_\delta(|v|^2) v_{xx} + \bigl(g_\delta(|v|^2)\bigr)_x v_x  + \bigl((g_\delta(|v|^2))_x v\bigr)_x.
$$
Since
$$
\bigl| g_\delta(|v|^2) - g_\delta(|\widetilde v|^2) \bigr| \leq |v -\widetilde v|
$$
it follows similarly to \eqref{3.5} that
$$
\bigl| g_\delta(|v|^2) v_{xx} - g_\delta(|\widetilde v|^2) \widetilde v_{xx}\bigr| \leq (1+|v|)|\cdot |v_{xx} - \widetilde v_{xx}| + |\widetilde v_{xx}|\cdot |v -\widetilde v|. 
$$
In comparison with \eqref{3.6} and \eqref{3.7} other terms are not estimated uniformly with respect to $\delta$. Since 
$$
\Bigl| \frac1{(|v|^2+\delta)^{1/2}} - \frac1{(|\widetilde v|^2+\delta)^{1/2}} \Bigr| \leq \frac{|v- \widetilde v|}{(|v|^2+\delta)^{1/2}(|\widetilde v|^2+\delta)^{1/2}}
$$
we derive with the use of \eqref{B.5} that
$$
\bigl| \bigl(g_\delta(|v|^2)\bigr)_x v_x - \bigl(g_\delta(|\widetilde v|^2)\bigr)_x \widetilde v_x\bigr| \leq
c\Bigl[\bigl(|v_x| +|\widetilde v_x|\bigr) |v_x - \widetilde v_x| + \frac1{\delta^{1/2}} \bigl(|v_x|^2 + |\widetilde v_x|^2\bigr) |v -\widetilde v|\Bigr].
$$
Next, since
$$
\Bigl| \frac1{(|v|^2+\delta)^{3/2}} - \frac1{(|\widetilde v|^2+\delta)^{3/2}} \Bigr| \leq
\frac{3|v- \widetilde v| \bigl(|v|^2 +\delta + |\widetilde v|^2 +\delta\bigr)}{2(|v|^2+\delta)^{3/2}(|\widetilde v|^2+\delta)^{3/2}}
$$
we derive with the use of \eqref{B.6} that
\begin{multline*}
\bigl| \bigl((g_\delta(|v|^2))_x v\bigr)_x - \bigl((g_\delta(|\widetilde v|^2))_x \widetilde v\bigr)_x \bigr| \leq 
c\Bigl[\bigl(|v| +|\widetilde v|\bigr) |v_{xx} - \widetilde v_{xx}| +\bigl(|v_x| +|\widetilde v_x|\bigr) |v_x - \widetilde v_x|  \\
+ \bigl(|v_{xx}| +|\widetilde v_{xx}|\bigr) |v - \widetilde v| 
+ \frac1{\delta^{1/2}} \bigl(|v_x| +|\widetilde v_x|\bigr)\bigl(|v| +|\widetilde v|\bigr) |v_x - \widetilde v_x|
+ \frac1{\delta^{1/2}} \bigl(|v_x|^2 + |\widetilde v_x|^2\bigr) |v -\widetilde v| \\
+ \frac1{\delta} \bigl(|v_x|^2 +|\widetilde v_x|^2\bigr)\bigl(|v| +|\widetilde v|\bigr) |v -\widetilde v| \Bigr].
\end{multline*}
Other terms are estimated in a similar and more obvious way. Then similarly to \eqref{3.8}, \eqref{3.9} and with the use of the embedding $C([0,T];H^1) \subset C_b(\overline{\Pi}_T)$ it follows that for $v, \widetilde v \in \widetilde X_R(\Pi_{T})$
\begin{equation}\label{3.15}
\|F(v) - F(\widetilde v)\|_{L_2(0,T; H^1)} \leq c \Bigl[(1+2R) +\frac{(2R)^3}{\delta}\Bigr] \|v -\widetilde v\|_{\widetilde X_R(\Pi_{T})},
\end{equation}
whence similarly to \eqref{3.13} follows that
\begin{equation}\label{3.16} 
\|u - \widetilde u\|_{X(\Pi_{T})} \leq c_1 T^{1/2} \Bigl[(1+2R) +\frac{(2R)^3}{\delta}\Bigr]
\|v - \widetilde v\|_{X(\Pi_{T})}.
\end{equation}
where the constant $c_1$ does not depend on $\delta$. Then inequality \eqref{3.16} ensures that for
$T(\delta) = \min\bigl[T_0,\bigl(2c_1 (1+2R +(2R)^3/\delta)\bigr)^{-2}\bigr]$
\begin{equation}\label{3.17}
\|\Lambda v - \Lambda\widetilde v\|_{X(\Pi_{T})} \leq \frac12 \|v - \widetilde v\|_{X_{(\Pi_T})}.
\end{equation}.

Next, assume that for certain $T<T_0$ there constructed a unique solution $u\in \widetilde X_R(\Pi_T)$ to problem \eqref{3.1}, \eqref{1.2}. Choose
$$
t_0 = \min(T_0-T, T(\delta)),\quad T_1 =T+t_0.
$$
Let $\Pi_{T,T_1} = (T,T_1)\times \mathbb R$, define the space $X(\Pi_{T,T_1})$ similarly to $X(\Pi_T)$ ($0$ is substituted by $T$, $T$ --- by $T_1$) and let $\widetilde X_R(\Pi_{T,T_1}) = \{v\in X(\Pi_{T,T_1}): \|v\|_{X(\Pi_{T,T_1})} \leq R, v\big|_{t=T} = u\big|_{t=T}\}$. For any $v\in \widetilde X_R(\Pi_{T,T_1})$ let
$$
V(t,x) \equiv \begin{cases} u(t,x),\quad & t\in [0,T],\\ v(t,x),\quad & t\in (T,T_1]. \end{cases}
$$
Note that $V\in \widetilde X_{2R}(\Pi_{T_1})$. On the strip $\Pi_{T_1}$ consider a map $\Lambda_1$, where $U = \Lambda_1 V \in X(\Pi_{T_1})$ is a solution to problem \eqref{3.11}, where $v$ is substituted by $V$. By virtue of Lemmas~\ref{L2.6} and~\ref{L3.1} such a map exists and
$$
\|U\|_{X(\Pi_{T_1})} \leq c_0 M + c_0 T_0^{1/2} 2R (1+2 R) \leq R.
$$
Note that $U(t,x) = u(t,x)$ for $t\leq T$ because of uniqueness. Define a map $\Lambda$ on the set $\widetilde X_R(\Pi_{T,T_1})$ in the following way: let $\Lambda v$ be the restriction of $\Lambda_1 V$ on the strip $\Pi_{T,T_1}$. Note that $\Lambda v \in \widetilde X_R(\Pi_{T,T_1})$. Moreover, similarly to \eqref{3.15}--\eqref{3.17}
\begin{multline*}
\|\Lambda v - \Lambda\widetilde v\|_{X(\Pi_{T,T_1})} \leq c_1 t_0^{1/2} \Bigl[(1+2R) +\frac{(2R)^3}{\delta}\Bigr]
\|v - \widetilde v\|_{X(\Pi_{T,T_1})} \\ \leq \frac12 \|v - \widetilde v\|_{X(\Pi_{T,T_1})}.
\end{multline*}
As a result, the map $\Lambda$ is a contraction. Let $v_0$ be the unique fixed point of the map $\Lambda$. Extend the function $u$ to the strip $\Pi_{T_1}$ by the function $v_0$. Then this extended function $u$ is the fixed point of the map $\Lambda_1$, in particular, $u\in \widetilde X_R(\Pi_{T_1})$.

Moving step by step we obtain the solution $u= u_\delta\in \widetilde X_R(\Pi_{T_0})$ to problem \eqref{3.1}, \eqref{1.2}.
\end{proof}

\begin{theorem}\label{T3.1}
Let $u_0\in H^1$, $d\in H^1$. Then there exist $T_0>0$, depending on $\|u_0\|_{H^1}$, and a solution $u\in X(\Pi_{T_0})$ to problem \eqref{1.1}, \eqref{1.2}.
\end{theorem}

\begin{proof}
Let $T_0$ be the same value as obtained in Lemma \ref{L3.2}.

For any $\delta\in (0,1]$ consider the solution $u_\delta \in X(\Pi_{T_0})$ to problem \eqref{3.1}, \eqref{1.2}. According to \eqref{3.10} this set is bounded in the space $X(\Pi_{T_0})$. In particular, for any interval $I_n = (-n,n)$, $n\in \mathbb N$, uniformly with respect to $\delta$
\begin{equation}\label{3.18}
\|u_\delta\|_{L_2(0,T;H^2(I_n))} \leq c(n).
\end{equation}
Then since
\begin{equation}\label{3.19}
\bigl( g_\delta(|u|^2)\bigr)_x = g_\delta'(|u|^2) (u\bar u_x + u_x \bar u)
\end{equation}
with the use of equation \eqref{3.1} itself and inequalities \eqref{3.2} we derive that uniformly with respect to $\delta$
\begin{equation}\label{3.20}
\|u_{\delta t}\|_{L_2(0,T;H^{-1}(I_n))} \leq c(n).
\end{equation}
Then by the standard argument we select a sequence $\delta_k\to +0$ such that
\begin{align*}
u_{\delta_k} \rightharpoonup u&\quad  *-\text{weakly in}\quad L_\infty(0,T_0; H^1),\\
u_{\delta_k x} \rightharpoonup u_x&\quad  *-\text{weakly in}\quad L_6(0,T_0; L_\infty),\\
u_{\delta_k xx} \rightharpoonup u_{xx}&\quad  *-\text{weakly in}\quad L_\infty(\mathbb R; L_2(0,T_0)),\\
u_{\delta_k} \rightharpoonup u&\quad  *-\text{weakly in}\quad L_2(\mathbb R; L_\infty(0,T_0)),\\
u_{\delta_k}\rightarrow u &\quad \text{strongly in}\quad L_2(0,T_0;H^1(I_n)) \quad \forall n\in \mathbb N.
\end{align*}
Write the corresponding integral identities for the functions $u_\delta$:
\begin{multline}\label{3.21}
\int_0^T \!\! \int \bigl[ i u_\delta \phi_t -a u_{\delta} \phi_{xx} + ib u_{\delta}\phi_x +i u_{\delta} \phi_{xxx}-\lambda g_\delta(|u_\delta|^2|) u_\delta \phi  +i\beta g_\delta(|u_\delta|^2) u_\delta \phi_x \\   -idu_\delta\phi\bigr]\,dxdt  +
\int u_0 \phi\big|_{t=0}\,dx =0,
\end{multline}
where $\phi \in C^1([0,T];L_2) \cap C([0,T]; H^3)$, $\phi\big|_{t=T} =0$. Assume first that the function $\phi$ has a compact support with respect to $x$. We have:
$$
g_\delta(|u_\delta|^2) - |u| = (|u_\delta|^2 +\delta)^{1/2} - |u_\delta| +  |u_\delta| - |u|,
$$
where
$$
\bigl| (|u_\delta|^2 +\delta)^{1/2} - |u_\delta| \bigr| \leq \delta^{1/2}, \quad \bigl| |u_\delta| - |u|\bigr| \leq |u_\delta - u|,
$$
and the passage to the limit in \eqref{3.21} when $\delta = \delta_k$ is obvious. As a result we obtain the corresponding integral identity of \eqref{1.3}-type for the functions $\phi$ with a compact support. Approximating general test functions from Definition~\ref{D1.1} by such ones and passing to the limit we derive the desired integral identity in its general form.

Finally, note that according to Lemma~\ref{L3.1} $F(u(t,x)) \in L_1(0,T_0;H^1)$ and by virtue of Lemma~\ref{L2.6} and uniqueness of a solution to a linear problem $u\in X(\Pi_{T_0})$. 
\end{proof}

\begin{lemma}\label{L3.3}
Let $u_0\in H^1$, $d\in H^1$ and for certain $T>0$ a function $u\in X(\Pi_T)$ be a solution in $\Pi_T$ to problem \eqref{1.1}, \eqref{1.2}. Then for any $t\in (0,T]$
\begin{equation}\label{3.22}
\int |u(t,x)|^2\, dx  + 2\int_0^t \!\! \int d(x) |u(\tau,x)|^2\, dx d\tau = \int |u_0|^2\, dx .
\end{equation}
\end{lemma}

\begin{proof}
Consider $u$ as the solution to problem \eqref{2.1}, \eqref{2.2}, where $f(t,x) \equiv F(u(t,x))$ is defined by formula \eqref{3.3} for $\delta=0$. Then Lemma~\ref{L3.1} implies that $f\in L_2(0,T;H^1)$ and according to Lemma~\ref{L2.3} equality \eqref{2.15} holds:
\begin{equation}\label{3.23}
\frac{d}{dt} \int |u|^2\,dx +2\Im \int  \bigl[ \lambda |u| u +i \beta \bigl(|u| u\bigr)_x + i d(x) u\bigr] \bar u\,dx =0.
\end{equation}
Obviously,
$$
\Im \int |u| u \bar u \,dx =0,
$$
$$
\Im i\int 2 |u| u_x \bar u \,dx = \int 2 |u| \Re( \bar u  u_x) \, dx  = \frac2{3} \int \bigl(|u|^{3}\bigr)_x \, dx =0,
$$
$$
\Im i\int 2 |u|_x u \bar u \,dx = -\int 2|u| \bigl(|u|^2\bigr)_x \,dx  =   -\frac4{3} \int \bigl(|u|^{3}\bigr)_x \, dx =0
$$
and \eqref{3.23} yields that 
$$
\frac{d}{dt} \int |u|^2\,dx + 2\int d(x) |u|^2 \,dx =0.
$$
\end{proof}

\begin{theorem}\label{T3.2}
Let $u_0 \in H^{1,\alpha}$ for certain $\alpha\geq 0$, $d\in H^1$, $\beta\ne 0$. Then there exists a function $u(t,x)$ such that $u\in X(\Pi_T) \cap C_w([0,T];H^{1,\alpha})$ for any $T>0$ and is a  solution to problem \eqref{1.1}, \eqref{1.2} in $\Pi_T$; moreover, $u_{xx} \in L_2(0,T;L_2^{\rho'_{\alpha,\epsilon}(x)})$ for any $\epsilon>0$.
\end{theorem}

\begin{proof}
Estimate the following a priori estimate. Let for certain $T' \in (0,T]$ a function $u\in X(\Pi_{T'})$ be a solution to problem \eqref{1.1}, \eqref{1.2}. Then there exists a constant, depending on $T$, $\|u_0\|_{H^1}$ and $\|d\|_{H^1}$, such that
\begin{equation}\label{3.24}
\|u\|_{C([0,T'];H^1)} \leq c.
\end{equation}
First note that equality \eqref{3.22} yields that 
\begin{equation}\label{3.25}
\|u\|_{C[0,T']; L_2)} \leq c(T,\|u_0\|_{L_2},\|d\|_{L_\infty}).
\end{equation}
Next, let the function $F(u)$ be defined by formula \eqref{3.3} for $\delta=0$. Then Lemma~\ref{L3.1} provides that $F\in L_2(0,T';H^1)$ and, therefore, Lemma~\ref{L2.6} for $f\equiv F$ can be applied. Write down equality \eqref{2.24}:
\begin{equation}\label{3.26}
\frac{d}{dt} \int |u_x|^2\,dx +2\Im \int  \bigl[ \lambda |u| u +i \beta \bigl(|u| u\bigr)_x +  i d(x) u\bigr]_x \bar u_x\,dx =0.
\end{equation}
Here
\begin{multline*}
2 \Im \int \bigl(|u| u\bigr)_x \bar u_x \,dx = 2 \Im \int |u|_x u \bar u_x \,dx  + 2 \Im \int |u| u_x \bar u_x \,dx \\ =
2\Im \int |u|_x  u \bar u_x \,dx,
\end{multline*}
\begin{multline*}
2 \Im i\int \bigl(|u| u_x\bigr)_x \bar u_x \,dx = 2 \int |u|_x |u_x|^2 \,dx + 2 \Re \int |u| u_{xx} \bar u_x \,dx  \\ =
2 \int |u|_x |u_x|^2 \,dx + \int |u| \bigl(|u_x|^2\bigr)_x \,dx = \int |u|_x |u_x|^2 \,dx,
\end{multline*}
\begin{multline*}
2 \Im i\int (|u|_x u)_x \bar u_x \,dx = 2 \Re \int (|u|_x u)_x \bar u_x \,dx = -2\Re \int |u|_x u \bar u_{xx} \,dx \\ = 
-\int |u|_x (u \bar u_{xx} +u_{xx} \bar u)\, dx =
-\int |u|_{x} \bigl[\bigl(|u|^2\bigr)_{xx} -2|u_x|^2\bigr] \,dx \\ = -\int |u|_{x} \bigl(|u|^2\bigr)_{xx} \,dx
+ 2 \int |u|_x |u_x|^2 \,dx.
\end{multline*}
As a result it follows from \eqref{3.26} that
\begin{multline}\label{3.27}
\frac{d}{dt} \int |u_x|^2\,dx + 2\lambda \Im \int |u|_x  u \bar u_x \,dx + 3\beta \int |u|_x |u_x|^2 \,dx \\ -
\beta \int |u|_{x} \bigl(|u|^2\bigr)_{xx} \,dx +2 \int d(x) |u_x|^2 \,dx + 2\Re \int d'(x) u \bar u_x \,dx =0.
\end{multline}

Write down equality \eqref{2.26}:
\begin{equation}\label{3.28}
i\frac{d}{dt} \int u \bar u_x\, dx + 2 \Re \int \bigl[ \lambda |u| u +i \beta \bigl(|u| u\bigr)_x +  i d(x) u\bigr] \bar u_x\, dx =0.
\end{equation}
Here
$$
2\Re \int |u| u \bar u_x \,dx = \int |u| \bigl(|u|^2\bigr)_x \,dx =0,
$$
$$
2\Re i\int |u| u_x \bar u_x \,dx = - 2\Im \int |u| |u_x|^2 \,dx =0,
$$
$$
2\Re i\int |u|_x u \bar u_x \,dx = -2\Im \int |u|_x u \bar u_x \,dx.
$$
As a result it follows from \eqref{3.28} that
\begin{equation}\label{3.29}
i\frac{d}{dt} \int u \bar u_x\, dx - 2\beta \Im \int |u|_x u \bar u_x \,dx - 2\Im \int d(x) u \bar u_x \,dx =0.
\end{equation}

Finally, write down equality \eqref{2.25}:
\begin{multline}\label{3.30}
\frac2{3}\frac{d}{dt} \int |u|^{3}\,dx + 2a \Im \int |u|_x u \bar u_x \,dx + 3\int |u|_x |u_x|^2\, dx  \\ -
\int |u|_x \bigl(|u|^2\bigr)_{xx}\, dx   
+ 2\Im \int \bigl[ \lambda |u| u +i \beta \bigl(|u| u\bigr)_x +  i d(x) u\bigr] |u| \bar u\, dx =0.
\end{multline}
Here
$$
2\Im \int |u|^{2} u \bar u \,dx =0,
$$
$$
2\Im i\int |u|^{2} u_x \bar u \,dx  =  \int |u|^{2} \bigl(|u|^2\bigr)_x \,dx =0,
$$
$$
2\Im i\int |u|_x |u| u \bar u \,dx = \int \bigl(|u|^{2}\bigr)_x |u|^2 \,dx =0.
$$
As a result, it follows from \eqref{3.30} that
\begin{multline}\label{3.31}
\frac2{3}\frac{d}{dt} \int |u|^{3}\,dx + 2a \Im \int |u|_x u \bar u_x \,dx + 3\int |u|_x |u_x|^2\, dx \\ -
\int |u|_x \bigl(|u|^2\bigr)_{xx}\, dx 
+ 2\int  d(x) |u|^{3} \, dx =0.
\end{multline}

Sum equality \eqref{3.27} with \eqref{3.29}, multiplied by $\displaystyle \left( \frac{\lambda}{\beta} - a\right)$, and with \eqref{3.31}, multiplied by $\displaystyle -\beta$, then
\begin{multline}\label{3.32}
\frac{d}{dt} \left[ \int |u_x|^2 \,dx + i\left(\frac{\lambda}{\beta} - a\right) \int u \bar u_x \,dx  - 
\frac{2\beta}{3 } \int |u|^{3} \,dx \right] \\ 
+ 2\int d(x)|u_x|^2 \,dx  + 2\Re \int d'(x) u \bar u_x \,dx  -2
\left(\frac{\lambda}{\beta} - a\right) \Im \int d(x) u \bar u_x \,dx \\ -
2\beta \int d(x) |u|^{3} \,dx =0.
\end{multline}
Note that 
\begin{multline*}
\Bigl| 2\Re \int d'(x) u \bar u_x \,dx \Bigr| \leq 2\sup\limits_{x\in\mathbb R} |u| \cdot\|d'\|_{L_2} \Bigl(\int |u_x|^2 \,dx \Bigr)^{1/2} \\ \leq
c \Bigl( \int |u_x|^2\, dx\Bigr)^{3/4} \Bigl( \int |u|^2\, dx\Bigr)^{1/4},
\end{multline*}
\begin{multline*}
\int |u|^{3}\, dx \leq 
\sup\limits_{x\in \mathbb R} |u| \int |u|^2\, dx \leq 
c \Bigl(\int |u_x|^2\,dx\Bigr)^{1/4} 
\Bigl(\int |u|^2\,dx\Bigr)^{5/4} \\ \leq
\varepsilon \int |u_x|^2\,dx +
c(\varepsilon) \Bigl(\int |u|^2\,dx\Bigr)^{5/3},
\end{multline*}
where $\varepsilon >0$ can be chosen arbitrarily small. Then equality \eqref{3.32} with the use of the already obtained estimate \eqref{3.25} provides estimate \eqref{3.24}. 

Global estimate \eqref{3.24} together with the results on local well-posedness from Theorem~\ref{T3.1} provides existence of the solution to problem \eqref{1.1}, \eqref{1.2} $u\in X(\Pi_T)$ $\forall T>0$. 

Next, establish the corresponding properties of the solution $u$ in weighted spaces.
Let $\psi(x)\in C_b^3$, $\psi, \psi' \geq 0$. Write down equality \eqref{2.17} for $f\equiv F$, then
\begin{multline}\label{3.33}
\frac{d}{dt}\int |u(t,x)|^2\psi\, dx +3 \int |u_x|^2\psi' \,dx = 2a \Im \int u_x \bar u \psi'\, dx \\ +
\int |u|^2 \psi'''\, dx +  b\int |u|^2 \psi' \, dx - 2\Im \int \bigl[ \lambda |u| u + i \beta \bigl(|u| u\bigr)_x + i d(x) u \bigr] \bar u \psi\, dx.
\end{multline}
Here
\begin{equation}\label{3.34}
\Im \int |u| u \bar u \psi \,dx =0,
\end{equation}
\begin{multline}\label{3.35}
-\Im i\int 2 \bigl(|u| u\bigr)_x \bar u \psi \,dx = - \int 2\Re \bigl(|u| u\bigr)_x \bar u \psi \,dx = \int |u| 2\Re (u \bar u_x)  \psi\,dx  \\ +  2 \int |u|^3 \psi'\,dx  =   
\int |u| \bigl(|u|^{2}\bigr)_x \psi \, dx 
 +  2 \int |u|^3 \psi'\,dx =  \frac43 \int |u|^3 \psi' \,dx.
\end{multline}
Note that 
\begin{equation}\label{3.36}
\int |u|^3 \psi'\,dx \leq  \sup\limits_{x\in \mathbb R} \bigl(|u|\sqrt{\psi'}\bigr) \Bigl(\int |u|^2 \psi' \,dx \int |u|^2 \,dx \Bigr)^{1/2},
\end{equation}
where
\begin{multline}\label{3.37}
\sup\limits_{x\in \mathbb R} |u|\sqrt{\psi'} \leq \Bigl( \int \bigl| (u \sqrt{\psi'})_x u\sqrt{\psi'} \bigr| \,dx \Bigr)^{1/2}  \leq
\Bigl( \int \bigl( |u_x u| \psi' + \frac12 |u^2 \psi''| \bigr)\,dx \Bigr)^{1/2} \\ \leq
\Bigl( \int |u_x|^2 \psi' \,dx \int |u|^2\psi' \,dx \Bigr)^{1/4} + \Bigl( \int |u|^2 |\psi''| \,dx \Bigr)^{1/2}.
\end{multline}
As a result,
\begin{multline}\label{3.38}
\int |u|^3 \psi' \, dx \leq \varepsilon \int |u_x|^2\psi' \,dx + c(\varepsilon) \Bigl(\int |u|^2 \,dx \Bigr)^{2/3} \int |u|^2\psi' \,dx  \\ + 
c \Bigl( \int |u|^2\,dx \Bigr)^{1/2} \int |u|^2 (\psi' +|\psi''|)\,dx,
\end{multline}
where $\varepsilon>0$ can be chosen arbitrarily small. Then equality \eqref{3.33} and estimate \eqref{3.25} imply that
\begin{equation}\label{3.39}
\frac{d}{dt}\int |u(t,x)|^2\psi\, dx +\int |u_x|^2\psi' \,dx \leq c\int |u|^2 \bigl(\psi+ \psi' + |\psi''| + |\psi'''|\bigr)\,dx,
\end{equation}
where the constant $c$ depends on $T$, $\|u_0\|_{L_2}$, $\|d\|_{L_\infty}$, $a$, $b$, $\beta$ and does not depend on $\psi$.

Approximate the function $\rho_{\alpha,\epsilon}(x)$ in such a way: for any $r>0$ let
\begin{equation}\label{3.40}
\rho_{r,\alpha,\epsilon}(x) \equiv \rho_{\alpha,\epsilon}(x) \eta(x+r-1) + \rho_{\alpha,\epsilon}(r+1) \eta(x-r),
\end{equation}
where the cut-off function $\eta$ is defined in Section~\ref{S1}.

Obviously, $\rho_{r,\alpha,\epsilon}(x) = \rho_{\alpha,\epsilon}(x)$ for $x\leq r$, $\rho_{r,\alpha,\epsilon}(x) = (2+r)^{2\alpha}$ for $x\geq r+1$. The function $\rho_{r,\alpha,\epsilon} \in C_b^\infty$, $\rho^{(j)}_{r,\alpha,\epsilon}(x) \to \rho^{(j)}_{\alpha,\epsilon}(x)$ when $r\to +\infty$ and $|\rho^{(j)}_{r,\alpha,\epsilon} (x)| \leq c(\alpha,\epsilon,j) \rho_{\alpha,\epsilon} (x)$ for any $x\in \mathbb R$ and any non-negative integer number $j$, $0 \leq\rho'_{r,\alpha,\epsilon}(x) \leq c(\alpha) \rho'_{\alpha,\epsilon}(x)$ for any $x\in \mathbb R$.

Set in \eqref{3.39} $\psi(x) \equiv \rho_{r,\alpha,\epsilon}(x)$. then uniformly with respect to $r$
\begin{equation}\label{3.41}
\frac{d}{dt}\int |u(t,x)|^2\rho_{r,\alpha,\epsilon}\, dx +\int |u_x|^2\rho_{r,\alpha,\epsilon}' \,dx \leq c\int |u|^2 \rho_{r,\alpha,\epsilon}\,dx,
\end{equation}
whence it follows that
\begin{equation}\label{3.42}
\|u\|_{L_\infty(0,T;L_2^\alpha)} +\|u_x\|_{L_2(0,T;L_2^{\rho'_{\alpha,\epsilon}(x)})} \leq c(T,\alpha,\epsilon, \|u_0\|_{L_2^\alpha}, \|d\|_{L_\infty}, a, b, \beta).
\end{equation}

Again let $\psi(x) \equiv \rho_{r,\alpha,\epsilon}(x)$. Write down corresponding equality \eqref{2.28}. Note that
\begin{equation}\label{3.43}
F_x = -\lambda \bigl(|u| u\bigr)_x - i \beta \bigl(|u| u\bigr)_{xx} - i \bigl(d(x) u\bigr)_x,
\end{equation}
then
\begin{multline*}
\frac{d}{dt}\int |u_x(t,x)|^2\psi\, dx +3 \int |u_{xx}|^2\psi' \,dx - 2a \Im \int u_{xx} \bar u_x \psi'\, dx \\ -
\int |u_x|^2 \psi'''\, dx -b \int |u_x|^2 \psi' \, dx + 2 \lambda\Im \int |u|_x u \bar u_x \psi\, dx  \\ +
2\beta \Re \int \bigl(|u| u\bigr)_{xx} \bar u_x \psi \,dx + 2\Re \int \bigl(d(x) u\bigr)_x \bar u_x \psi \,dx =0.
\end{multline*}
Here
\begin{multline*}
2\Re \int \bigl(|u| u\bigr)_{xx} \bar u_x \psi \,dx = 2\Re \int \bigl(|u| u_x\bigr)_{x} \bar u_x \psi \,dx + 2\Re \int \bigl(|u|_x u\bigr)_{x} \bar u_x \psi \,dx \\ =
2\int |u|_x |u_x|^2\psi \,dx + 2\Re \int  |u| u_{xx} \bar u_x \psi \,dx - 2\Re \int |u|_x u \bar u_{xx} \psi \,dx \\ -
2\Re \int |u|_x u \bar u_x \psi' \,dx =
2\int |u|_x |u_x|^2 \psi \,dx  + \int |u| \bigl(|u_x|^2\bigr)_x \psi\,dx \\ -
\int |u|_x \bigl( (|u|^2)_{xx} - 2 |u_x|^2 \bigr) \psi\,dx  - 2\Re \int |u|_x u \bar u_x \psi' \,dx \\ =
3 \int |u|_x |u_x|^2 \psi\,dx  - \int |u|_x \bigl( |u|^2\bigr)_{xx} \psi\,dx - \int |u| |u_x|^2 \psi' \,dx  -
\int |u|_x (|u|^2)_x \psi' \,dx.
 \end{multline*}
As a result,
\begin{multline}\label{3.44}
\frac{d}{dt}\int |u_x(t,x)|^2\psi(x)\, dx +3 \int |u_{xx}|^2\psi' \,dx - 2a \Im \int u_{xx} \bar u_x \psi'\, dx \\ -
 \int |u_x|^2 \psi'''\, dx -b \int |u_x|^2 \psi' \, dx + 2 \lambda\Im \int |u|_x u \bar u_x \psi\, dx  \\ +
3\beta \int |u|_x |u_x|^2 \psi\,dx  - \beta \int |u|_x \bigl( |u|^2\bigr)_{xx} \psi\,dx  -\beta \int |u| |u_x|^2 \psi' \,dx \\ -
\beta \int |u|_x (|u|^2)_x \psi' \,dx + 2 \int d(x) |u_x|^2 \psi \,dx + 2\Re \int d'(x) u \bar u_x \psi \,dx =0.
\end{multline}
Next write down corresponding equality \eqref{2.29}:
\begin{multline*}
\frac23 \frac{d}{dt} \int |u|^3\psi \,dx + 2a \Im \int \bigl(|u|\psi\bigr)_x u \bar u_x \,dx + 
3\int \bigl(|u|\psi\bigr)_x |u_x|^2\, dx \\ -
\int \bigl(|u|\psi\bigr)_x \bigl(|u|^2\bigr)_{xx}\, dx -\frac23 b \int |u|^3\psi' \,dx 
+2 \lambda \Im \int |u|^4 \psi\, dx \\+
2 \beta \Re \int \bigl(|u| u\bigr)_{x} |u| \bar u \psi \,dx  + 2 \int d(x)  |u|^3 \psi \,dx =0.
\end{multline*}
Here
\begin{multline*}
2 \Re \int \bigl(|u| u\bigr)_{x} |u| \bar u \psi \,dx  = 2\Re \int |u|^2 u_x \bar u\psi \,dx + 2 \int |u|_x |u|^3 \psi \,dx \\ =
2\int |u|^2 \bigl(|u|^2\bigr)_x \psi \,dx  = - \int |u|^4 \psi' \,dx. 
\end{multline*}
As a result,
\begin{multline}\label{3.45}
\frac23 \frac{d}{dt} \int |u|^3\psi \,dx + 2a \Im \int \bigl(|u|\psi\bigr)_x u \bar u_x \,dx + 
3\int \bigl(|u|\psi\bigr)_x |u_x|^2\, dx \\ -
\int \bigl(|u|\psi\bigr)_x \bigl(|u|^2\bigr)_{xx}\, dx -\frac23 b \int |u|^3\psi' \,dx 
-\beta \int |u|^4 \psi' \,dx  + 2 \int d(x)  |u|^3 \psi \,dx =0.
\end{multline}
Sum equality \eqref{3.44} with equality \eqref{3.45} multiplied by $-\beta$, then
\begin{multline}\label{3.46}
\frac{d}{dt} \int \left[ |u_x|^2 - \frac{2\beta}{3} |u|^3\right] \psi\,dx + 3 \int |u_{xx}|^2\psi' \,dx  - 2a \Im \int u_{xx} \bar u_x \psi' \,dx \\ -
\int |u_x|^2 \psi''' \,dx - b \int |u_x|^2 \psi' \,dx + 2\bigl(\lambda - a \beta\bigr) \Im \int |u|_x u \bar u_x \psi\,dx  \\ -
2a \beta \Im \int |u| u \bar u_x \psi' \,dx - 4 \beta \int |u| |u_x|^2 \psi' \,dx + \beta \int |u| \bigl(|u|^2\bigr)_{xx} \psi' \,dx \\ - 
\beta \int |u|_x (|u|^2)_x \psi' \,dx + \frac{2 b\beta}{3} \int |u|^3 \psi' \,dx + \beta^2 \int |u|^4 \psi' \,dx \\ + 
2\int d(x) |u_x|^2 \psi \,dx + 2\Re \int d'(x) u \bar u_x \psi \,dx - 2\beta \int d(x)|u|^3 \psi\,dx =0.
\end{multline}
The already obtained estimate \eqref{3.24} yields, in particular, that
\begin{equation}\label{3.47} 
\|u\|_{L_\infty(\Pi_T)} \leq c.
\end{equation}
Therefore, uniformly with respect to $r$ by virtue of \eqref{3.42}
$$
\sup\limits_{t\in (0,T)} \int |u|^3 \psi \,dx \leq c,  \quad \int |u|^2 |u_{xx}| \psi' \,dx \leq \varepsilon \int |u_{xx}|^2 \psi' \,dx + c(\varepsilon),
$$
where $\varepsilon>0$ can be chosen arbitrarily small; similarly to \eqref{3.37}
\begin{multline*}
\Bigl| \Re \int d'(x) u \bar u_x \psi \,dx \Bigr| \leq \sup\limits_{x\in \mathbb R} \bigl(|u| \sqrt{\psi}\bigr) \Bigl(\int (d'(x))^2 \,dx \int |u_x|^2 \psi \,dx \Bigr)^{1/2} \\ \leq 
c\int |u_x|^2 \psi \,dx.
\end{multline*}
Therefore, it follows from \eqref{3.46} that uniformly with respect to $r$
$$
\sup\limits_{t\in (0,T)} \int |u_x|^2 \psi \,dx + \int_0^T \!\! \int |u_{xx}|^2 \psi' \,dx dt \leq c,
$$ 
whence passing to the limit when $r\to +\infty$ we obtain that 
\begin{equation}\label{3.48}
\|u_x\|_{L_\infty(0,T;L_2^\alpha)} + \|u_{xx}\|_{L_2(0,T;L_2^{\rho'_{\alpha,\epsilon}(x)})} \leq c.
\end{equation}
The end of the proof is standard.
\end{proof}

\begin{theorem}\label{T3.3}
Let $u_0\in L_2^\alpha$ for certain $\alpha\geq 0$, $d\in L_\infty$, then there exists a function $u(t,x)$ such that for any $T>0$ it is a weak solution to problem \eqref{1.1}, \eqref{1.2}, $u\in C_w([0,T];L_2^\alpha)$, verifying  properties 
\begin{equation}\label{3.49}
\sigma(u;T) = \sup\limits_{x_0\in \mathbb R} \int_0^T \!\int_{x_0}^{x_0+1} |u_x|^2 \,dxdt \leq c(T,\|u_0\|_{L_2}, \|d\|_{L_\infty})
\end{equation} 
and $u_x \in L_2(0,T; L_2^{\rho'_{\alpha,\epsilon}(x)})$ for any $\epsilon>0$.
\end{theorem}

\begin{proof}
Apply the mollifying procedure to approximate the functions $u_0$ and $d$ by smooth ones: $u_{0h}, d_h \in H^1$, $u_{0h} \to u_0$ in $L_2^\alpha$, $d_h \to d$ in $L_2$ when $h\to +0$, $\|d_h\|_{L_\infty} \leq \|d\|_{L_\infty}$. Consider a set of regularized problems
\begin{gather}\label{3.50}
i u_t +a u_{xx} + i b u_x + iu_{xxx} +\lambda |u| u +i\beta_{h} \bigl(|u| u\bigr)_x  + id_h(x)u=0,
\\ \label{3.51}
u(0,x) = u_{0h}(x),\quad x\in \mathbb R.
\end{gather} 
where $\beta_{h} = \beta$ if $\beta\ne 0$, $\beta_{h} =h$ if $\beta=0$. According to Theorem~\ref{T3.2} for any $h\in (0,1]$ there exists a function $u_h \in X(\Pi_T)\cap C_w([0,T];H^{1,\alpha})$ for any $T>0$ which satisfies \eqref{3.50}, \eqref{3.51} in $\Pi_T$. Equality \eqref{3.22} yields that for any $T>0$ uniformly with respect to $h$
\begin{equation}\label{3.52}
\|u_h\|_{L_\infty(0,T;L_2)} \leq c(T).
\end{equation}
Write for the functions $u_h$ corresponding inequality \eqref{3.39} for $\psi(x) \equiv \rho_{0,1}(x-x_0)$ for any $x_0\in \mathbb R$, then the properties of the function $\rho_{0,1}$ and estimate \eqref{3.52} provide that uniformly with respect to $h$
\begin{equation}\label{3.53}
\sigma (u_h;T) \leq c(T).
\end{equation}
Moreover, using in \eqref{3.39} $\psi(x) \equiv \rho_{r,\alpha,\epsilon}(x)$ similarly to \eqref{3.41}, \eqref{3.42} we derive that uniformly with respect to $h$
\begin{equation}\label{3.54}
\|u_h\|_{L_\infty(0,T;L_2^\alpha)} + \|u_{hx}\|_{L_2(0,T;L_2^{\rho'_{\alpha,\epsilon}(x)})} \leq c(T,\alpha,\epsilon).
\end{equation}

It follows from the well--known embedding $L_1 \subset H^{-1}$, estimate \eqref{3.52} and equation \eqref{3.50} itself that uniformly with respect to $h$
\begin{equation}\label{3.55}
\|u_{ht}\|_{L_\infty(0,T;H^{-3})} \leq c(T).
\end{equation}
Then by the standard argument we select a sequence $h_k\to +0$ such that for any $T>0$
\begin{align*}
u_{h_k} \rightharpoonup u&\quad  *-\text{weakly in}\quad L_\infty(0,T; L_2),\\
u_{h_k x} \rightharpoonup u_x&\quad  \text{weakly in}\quad L_2(0,T; L_2(I_n)),\\
u_{h_k}\rightarrow u &\quad \text{strongly in}\quad L_2(0,T;L_2(I_n)) \quad \forall n\in \mathbb N, I_n =(-n,n).
\end{align*}
Write down the corresponding integral identities for the functions $u_h$ and certain $T>0$:
\begin{multline}\label{3.56}
\iint_{\Pi_T} \bigl[ i u_h \phi_t  -a u_{h} \phi_{xx}+ i b u_{h}\phi_x +i u_{h} \phi_{xxx} -\lambda |u_h| u_h \phi  +i\beta_{h} |u_h| u_h \phi_x \\   -
d_h(x) u_h\phi\bigr]\,dxdt  + \int u_{0h} \phi\big|_{t=0}\,dx =0,
\end{multline}
where $\phi \in C^1([0,T];L_2) \cap C([0,T]; H^3)$, $\phi\big|_{t=T} =0$. Assume first that the function $\phi$ has a compact support with respect to $x$. Then passage to the limit in \eqref{3.56} when $h= h_k \to +0$ is obvious and we obtain \eqref{1.3} for such test functions. Approximating general test functions from Definition~\ref{D1.1} by such ones and passing to the limit we derive integral identity \eqref{1.3} in its general form.

The end of the proof is standard.
\end{proof}

\begin{remark}\label{R3.1}
Let $u_0\in H^1$, $d\in H^1$ and $u\in X(\Pi_T)$ be a solution to problem \eqref{1.1}, \eqref{1.2} for certain $T>0$. Let $|\lambda|, |\beta| \leq A$ for certain $A>0$. Then similarly to \eqref{3.52}, \eqref{3.53} one can obtain that
\begin{equation}\label{3.57}
\|u\|_{L_\infty(0,T;L_2)} + \sigma(u;T) \leq c(T, \|u_0\|_{L_2}, \|d\|_{L_\infty}, A).
\end{equation}
\end{remark}

\section{Uniqueness and continuous dependence}\label{S4}

\begin{theorem}\label{T4.1}
Let $u_0, \widetilde u_0 \in L_2^{3/4}$, $d\in L_\infty$, $u, \widetilde u \in L_\infty(0,T;L_2^{3/4})$ -- be solutions to problem \eqref{1.1}, \eqref{1.2} (for $\widetilde u$ the initial function in \eqref{1.2} is equal to $\widetilde u_0$) in the strip $\Pi_T$ for certain $T>0$. Assume that $\|u\|_{L_\infty(0,T;L_2^{3/4})}, |\widetilde u\|_{L_\infty(0,T;L_2^{3/4})} \leq M$ for certain $M>0$ Then for any $\epsilon>0$ there exists a constant $c = c(T,M,\epsilon)$ such that
\begin{equation}\label{4.1}
\esssup\limits_{t\in (0,T)} \int \bigl| u(t,x) - \widetilde u(t,x) \bigr|^2 \rho_{3/4,\epsilon} \,dx \leq
c\int \bigl| u_0 - \widetilde u_0 \bigr|^2 \rho_{3/4,\epsilon} \,dx.
\end{equation}
In particular a weak solution to problem \eqref{1.1}, \eqref{1.2} in unique in the space $L_\infty(0,T;L_2^{3/4})$.
\end{theorem}

\begin{proof}
Let $w \equiv u - \widetilde u$, $w_0 \equiv u_0 - \widetilde u_0$, $f_{11} \equiv -i d w$, $f_{12} \equiv - \lambda \bigl(|u|u - |\widetilde u|\widetilde u \bigr)$, $f_2 \equiv  -i \beta \bigl(|u|u - |\widetilde u|\widetilde u \bigr)$. Then the function $w\in L_\infty(0,T;L_2^{3/4})$ is a weak solution to problem
\begin{equation}\label{4.2}
i w_t + a w_{xx} +i b w_x  + i w_{xxx} = f_{11} + f_{12} + f_{2x},\quad
w\big|_{t=0} = w_0.
\end{equation}
Consider a linear problem
\begin{equation}\label{4.3}
i w_{1t} + a w_{1xx} +i w_{1x}   + i w_{1xxx} = f_{11},\quad
w_1\big|_{t=0} = w_0.
\end{equation}
Note that $w_0 \in L_2^{3/4}$, $f_{11} \in L_\infty(0,T;L_2^{3/4})$. Then according to Lemma~\ref{L2.5} there exists a solution to this problem $w_1 \in C([0,T]; L_2^{3/4})$, $w_{1x} \in L_2(0,T; L_2^{\rho'_{3/4,\epsilon}(x)})$ for any $\epsilon>0$. Write down corresponding equality \eqref{2.22} for $\rho(x) \equiv \rho_{3/4,\epsilon}(x)$: 
\begin{multline}\label{4.4}
\frac{d}{dt}\int |w_1(t,x)|^2\rho(x)\, dx +3\int |w_{1x}|^2\rho' \,dx = 2a \Im \int w_{1x} \bar w_1 \rho'\, dx \\ +
 \int |w_1|^2 \rho'''\, dx + b \int |w_1|^2 \rho' \, dx - 2\Re \int d w \bar w_1 \rho\, dx. 
\end{multline}
It follows from \eqref{4.4} that
\begin{equation}\label{4.5}
\int |w_1(t,x)|^2\rho \,dx \leq c \int |w_0|^2 \rho \,dx + c \int_0^t \!\! \int |w|^2 \rho \,dx d\tau.
\end{equation}
Let $w_2 \equiv w - w_1$, then $w_2 \in L_\infty(0,T;L_2^{3/4})$ and is a solution to a problem
\begin{equation}\label{4.6}
i w_{2t} + aw_{2xx}  +i b w_{2x}  + i w_{2xxx}  = f_{12} + f_{2x},\quad
w_2\big|_{t=0} = 0.
\end{equation}
Note that $|f_{12}|, |f_2| \leq c\bigl( |u| + |\widetilde u|\bigr) |w|$; in particular, $f_{12}, f_2 \in L_\infty(0,T;L_1^{3/2})$. Write corresponding inequality \eqref{2.20}, then uniformly with respect to $t\in (0,T)$ for $\rho(x) \equiv \rho_{3/4,\epsilon}(x)$ and $p>4$
\begin{multline}\label{4.7}
\esssup\limits_{\tau \in (0,t)} \Bigl(\int |w_2(\tau,x)|^2\rho \,dx \Bigr)^{1/2} \\ \leq c \Bigl[\int_0^t \Bigl( \int \bigl( |u| + |\widetilde u|\bigr) |w| \rho^{1/2} (1+x_+)^{3/4} \,dx \Bigl)^p d\tau \Bigr]^{1/p} \\ \leq c_1 \left( \| u \|_{L_\infty(0,T;L_2^{3/4})} + \| \widetilde u \|_{L_\infty(0,T;L_2^{3/4})} \right) 
\Bigl[\int_0^t \Bigl( \int |w|^2 \rho \,dx \Bigl)^{p/2} d\tau \Bigr]^{1/p}.
\end{multline}
Inequalities \eqref{4.5} and \eqref{4.7} yield that
\begin{multline}\label{4.8}
\esssup\limits_{\tau \in (0,t)} \Bigl(\int |w(\tau,x)|^2\rho \,dx \Bigr)^{1/2} \leq c \Bigl( \int |w_0|^2 \rho \,dx \Bigr)^{1/2} \\ +
c \Bigl[\int_0^t \int |w|^2 \rho \,dx d\tau \Bigr]^{1/2} + c\Bigl[\int_0^t \Bigl( \int |w|^2 \rho \,dx \Bigl)^{p/2} d\tau \Bigr]^{1/p},
\end{multline}
whence \eqref{4.1} follows by the standard argument.
\end{proof}

\section{Internal regularity}\label{S5}

\begin{theorem}\label{T5.1}
Let $u_0 \in L_2^\alpha$ for certain $\alpha>1/4$, $d\in L_\infty$, $u\in L_\infty(0,T;L_2^\alpha)$ be a solution to problem \eqref{1.1}, \eqref{1.2} for certain $T>0$. Then (after possible correction on a zero measure set) $u\in C_b(\overline{\Pi}_T^{t_0, x_0})$ for any $t_0\in (0,T)$, $x_0 \in\mathbb R$, where $\Pi_T^{t_0,x_0} =  (t_0,T) \times (x_0,+\infty)$.
\end{theorem}

\begin{proof}
Consider the function $u$ as a solution to problem \eqref{2.1}, \eqref{2.2}, where $f\equiv F$ and the function $F$ is given by formula \eqref{3.3} for $\delta=0$. Let
\begin{equation}\label{5.1}
f_{11} \equiv -i d(x) u, \quad f_{12} \equiv -\lambda u|u|, \quad f_2 \equiv -i \beta u|u|,
\end{equation}
$f_1 \equiv f_{11} + f_{12}$. Then $f= f_1 + f_{2x}$. Note that
\begin{equation}\label{5.2}
\|f_{11}\|_{L_\infty(0,T;L_2^\alpha)} \leq \|d\|_{L_\infty} \|u\|_{L_\infty(0,T;L_2^\alpha)},
\end{equation}
\begin{equation}\label{5.3}
\|f_{12}\|_{L_\infty(0,T;L_1)}, \|f_2\|_{L_\infty(0,T;L_1^{1/4})} \leq c\|u\|^2_{L_\infty(0,T;L_2^\alpha)}.
\end{equation}
Thus the hypothesis of Theorem~\ref{T2.1} is verified. Write down equality \eqref{2.7}. For the first term in the right-hand side of \eqref{2.7} apply estimate \eqref{2.9}. Next, similarly to \eqref{2.10}
\begin{multline}\label{5.4} \displaystyle
\Bigl| \int_0^t \!\! \int G(t-\tau,x-y) f_{11}(\tau,y)\, dy d\tau \Bigr| \\ \leq
c(T,\alpha)  (1+x_-)^{\alpha} \int_0^t \frac{\|f_{11}(\tau,\cdot)\|_{L_2^{\alpha}}}{(t-\tau)^{1/3}} \,d\tau  \leq
c_1(T,\alpha)  (1+x_-)^{\alpha} \|f_{11}\|_{L_\infty(0,T;L_2^{\alpha})},
\end{multline}
similarly to \eqref{2.11}
\begin{multline}\label{5.5}
\Bigl| \int_0^t \!\! \int G(t-\tau,x-y) f_{12}(\tau,y)\, dy d\tau \Bigr|  \\ \leq
c(T) \int_0^t \frac1{(t-\tau)^{1/3}} \Bigl[ \int_x^{+\infty} (1+y-x)^{-1/4} |f_{12}(\tau,y)|\, dy \\+
\int_{-\infty}^x e^{-c_0(x-y)^{3/2}T^{-1/2}} |f_{12}(\tau,y)|\, dy \Bigr]\,d\tau  \leq
c_1(T) \|f_{12}\|_{L_\infty(0,T;L_1)},
\end{multline}
similarly to \eqref{2.12}
\begin{multline}\label{5.6}
\Bigl| \int_0^t \!\! \int G_x(t-\tau,x-y) f_2(\tau,y)\, dy d\tau \Bigr|  \\ \leq
c(T)  \int_0^t \frac1{(t-\tau)^{3/4}} \Bigl[ \int_x^{+\infty} (1+y-x)^{1/4} |f_2(\tau,y)|\, dy \\+
\int_{-\infty}^x e^{-c_0(x-y)^{3/2}T^{-1/2}} |f_2(\tau,y)|\, dy \Bigr]\,d\tau  \leq
c_1(T) (1+x_-)^{1/4} \|f_2\|_{L_\infty(0,T;L_1^{1/4})}.
\end{multline}
Moreover, the considered integrals are continuous with respect to $(t,x)\in \Pi_T$.
\end{proof}

\begin{theorem}\label{T5.2}
Let $u_0 \in L_2^\alpha$ for certain $\alpha\geq 1/2$, $d\in H^1$. Consider the function $u(t,x)$, constructed in Theorem~\ref{T3.3}, which is the weak solution to problem \eqref{1.1}, \eqref{1.2} in any strip $\Pi_T$. Then $u_x \in C_w([t_0,T]; L_2^{\rho_{\alpha-1/2,\epsilon}(x)})$, $u_{xx} \in L_2(t_0,T; L_2^{\rho'_{\alpha-1/2,\epsilon}(x)})$ for any $t_0\in (0,T)$, $\epsilon>0$.
\end{theorem}

\begin{proof}
Approximate the function $u_0$ by functions $u_{0h} \in H^{1,\alpha}$, let $\beta_{h}$ be the same as in the proof of Theorem~\ref{T3.3} and consider the corresponding solutions $u_h(t,x)$ to problem \eqref{1.1}, \eqref{1.2} constructed in Theorem~\ref{T3.2}: $u_h \in X(\Pi_T) \cap C_w([0,T];H^{1,\alpha})$, $u_{xx} \in L_2(0,T;L_2^{\rho_{\alpha-1/2,\epsilon}(x)})$ $\forall T>0$, $\forall \epsilon>0$. Estimate \eqref{3.54} is verified uniformly with respect to $h$. Fix $t_0 \in (0,T)$. In the consequent argument temporarily drop the index $h$.

Let $\rho(x) \equiv \rho_{\alpha-1/2,\epsilon}(x)$, for $r>0$ let $\psi(x) \equiv \rho_{r,\alpha-1/2,\epsilon}(x)$. Write down corresponding equality \eqref{3.46} and multiply it by $\eta^2(2t/t_0-1)$. Let $v(t,x) \equiv u(t,x) \eta(2t/t_0 -1)$ (for simplicity in the consequent argument $\varphi \equiv \eta(2t/t_0 -1)$). Note that
$$
\varphi^2 \frac{d}{dt} \int |u_x|^2 \psi \,dx  = \frac{d}{dt} \int |v_x|^2 \psi \,dx -2 \int |u_x|^2 \psi \varphi'\varphi \,dx,
$$
$$
\varphi^2 \frac{d}{dt} \int |u|^3\psi \,dx = \frac{d}{dt} \int |u v^2| \psi \,dx - 2 \int |u|^3 \psi \varphi'\varphi \,dx.
$$
Then 
\begin{multline}\label{5.7}
\frac{d}{dt} \int \left[ |v_x|^2 - \frac{2\beta}{3} |uv^2|\right] \psi\,dx + 3 \int |v_{xx}|^2\psi' \,dx - 2a \Im \int v_{xx} \bar v_x \psi' \,dx \\ -
\int |v_x|^2 \psi''' \,dx - b \int |v_x|^2 \psi' \,dx  + 2\bigl(\lambda - a\beta\bigr) \Im \int u |v|_x \bar v_x \psi\,dx \\ -
2a\beta \Im \int |u| v \bar v_x \psi' \,dx - 4 \beta \int |u| |v_x|^2 \psi' \,dx + \beta \int |u| \bigl(|v|^2\bigr)_{xx} \psi' \,dx \\ - 
\beta \int |v|_x |uv|_x \psi' \,dx + \frac{2 b \beta}{3} \int |uv^2| \psi' \,dx + \beta^2 \int |u^2v^2| \psi' \,dx \\ + 
2\int d(x) |v_x|^2 \psi \,dx + 2\Re \int d'(x) v \bar v_x \psi \,dx - 2\beta \int d(x)|uv^2| \psi\,dx \\ -
2\int |u_x|^2 \psi \varphi'\varphi \,dx  + \frac{4\beta}{3}\int |u|^3 \psi \varphi'\varphi \,dx=0.
\end{multline}
Integrate this equality with respect to $t$, pass to the limit when $r\to +\infty$ and again differentiate with respect to $t$, then for a.e. $t\in (0,T)$ equality \eqref{5.7} holds, where $\psi(x)$ is substituted by $\rho(x)$. Uniformly with respect to $h$ similarly to \eqref{3.36}--\eqref{3.38}
$$
\int |uv^2|\rho\,dx \leq \sup\limits_{x\in \mathbb R} \bigl( |v| \sqrt{\rho}\bigr) \Bigl( \int |u|^2 \,dx  \int |v|^2 \rho \,dx\Bigr)^{1/2}  \leq 
\varepsilon \int |v_x|^2 \rho \,dx + c(\varepsilon),
$$
where $\varepsilon>0$ can be chosen arbitrarily small. Next, note that $\rho(x)/\rho'(x) \leq c (1+x_+)^{2}$ (in fact, it is obvious for $x\leq 0$, while for $x\geq 0$ if $\alpha>0$ then $\rho(x)/\rho'(x) = c(1+x)$  and if $\alpha=0$ then $\rho(x)/\rho'(x) \leq c(\varepsilon)(1+x)^{1+\varepsilon}$ for any $\varepsilon>0$). Therefore, since $\alpha \geq 1/2$
(here and further uniformly with respect to $h$)
\begin{multline*}
\int |u v_x^2| \rho \,dx \leq \sup\limits_{x\in \mathbb R} \bigl( |v_x| (\rho'\rho)^{1/4}\bigr) \Bigl(\int |v_x|^2 \rho \,dx \int |u|^2 (\rho/\rho')^{1/2} \,dx \Bigr)^{1/2} \\ \leq c \Bigl( \int \bigl(|v_{xx}|^2\rho' + |v_x|^2 \rho \bigr) \,dx \Bigr)^{1/4} \Bigr(\int |v_x|^2 \rho \,dx \Bigr)^{1/2} \Bigl(\int (1+x_+) |u|^2 \,dx \Bigr)^{1/2} \\ \leq 
\varepsilon \int |v_{xx}|^2 \rho' \,dx + c(\varepsilon) \int |v_x|^2 \rho \,dx,
\end{multline*}  
where $\varepsilon>0$ can be chosen arbitrarily small. Next,
\begin{multline*}
\int |u v v_{xx}| \rho' \,dx \leq \sup\limits_{x\in \mathbb R} \bigl( |v| \sqrt{\rho'}\bigr) \bigl(\int |u|^2 \,dx \int |v_{xx}|^2 \rho'\,dx \Bigr)^{1/2} \\ \leq
\varepsilon \int |v_{xx}|^2 \rho' \,dx + c(\varepsilon) \int |v_x|^2 \rho \,dx +c(\varepsilon),
\end{multline*}
where $\varepsilon>0$ also can be chosen arbitrarily small. Finally,
$$
\iint_{\Pi_T} |u_x|^2 \rho \, dx dt  \leq c \iint_{\Pi_T} |u_x|^2 \rho'_{\alpha,\epsilon} \,dx dt \leq c_1,
$$
\begin{multline*}
\iint_{\Pi_T} |u|^3 \rho \,dx dt \leq \int_0^T \sup\limits_{x\in \mathbb R} \bigl( |u| \sqrt{\rho}\bigr) \Bigl( \int |u|^2 \,dx  \int |u|^2 \rho \,dx\Bigr)^{1/2} dt \\ \leq  c \iint_{\Pi_T}  |u_x|^2 \rho \,dx dt + c \leq c_1.
\end{multline*}
As a result, since $\eta(2t/t_0-1) =1$ for $t\geq t_0$ we find that uniformly with respect to $h$
\begin{equation}\label{5.8}
\|u_{hx}\|_{L_\infty(t_0,T;L_2^{\rho_{\alpha-1/2,\epsilon}(x)})} + \|u_{hxx}\|_{L_2(t_0,T;L_2^{\rho'_{\alpha-1/2,\epsilon}(x)})} \leq c(T,t_0).
\end{equation}
Application of the closure finishes the proof.
\end{proof}

\begin{theorem}\label{T5.3}
Let $u_0 \in L_2^\alpha$ for certain $\alpha>3/4$, $d\in H^1$, $u\in L_\infty(0,T;L_2^\alpha)$ be a solution to problem \eqref{1.1}, \eqref{1.2} for certain $T>0$. Then (after possible correction on a zero measure set) $u, u_x\in C_b(\overline{\Pi}_T^{t_0, x_0})$ for any $t_0\in (0,T)$, $x_0 \in\mathbb R$.
\end{theorem}

\begin{proof}
The considered solution lies in the class of uniqueness, therefore, according to Theorem~\ref{T5.2} $u\in C_w([0,T]; L_2^\alpha)$, $u_x \in C_w([t_0,T]; L_2^{\rho_{\alpha-1/2,\epsilon}(x)})$, $u_{xx} \in L_2(t_0,T;L_2^{\rho'_{\alpha-1/2,\epsilon}(x)})$ for any $t_0\in (0,T)$, $\epsilon>0$.

Fix $t_0\in (0,T)$, $x_0\in \mathbb R$. Let $v(t,x) \equiv u_x(t,x) \eta(x-x_0+2+ b_+T) \eta(2t/t_0-1)$ (in the consequent argument for simplicity $\eta \equiv \eta(x-x_0+2 +b_{+}T), \varphi \equiv \eta(2t/t_0 -1)$), then the function $v\in C_w([0,T];L_2^{\alpha-1/2})$ is a weak solution to problem \eqref{2.1}, \eqref{2.2} for $u_0 \equiv 0$ and $f\equiv f_1 +f_{2x}$, where
\begin{multline}\label{5.9}
f_1 = f_{11} \equiv i u_x \eta \varphi' + a (2u_{xx} \eta' + u_x \eta'')\varphi + i b u_x \eta' \varphi + i u_x \eta''' \varphi \\ +
\lambda |u| u \eta' \varphi + i \beta \bigl(|u| u \bigr)_x \eta' \varphi - i d(x) u \eta' \varphi,
\end{multline}
\begin{equation}\label{5.10}
f_2 \equiv - \bigl[ \lambda |u|u + i \beta \bigl( |u|u\bigr)_x + i d(x) u \bigr] \eta \varphi + 3 i u_{xx} \eta' \varphi.
\end{equation}
Note that $\supp \eta' = [x_0-2 -b_{+}T, x_0-1 -b_{+}T]$, $\supp \varphi = [t_0/2, +\infty)$. Then $f_1 \in L_2(0,T; L_2^{\alpha-1/2})$. Moreover, $f_2 \in L_2(0,T; L_1^{\alpha})$, since for $t\geq t_0$
\begin{multline*}
\int |u_x u| (1+x_+)^{\alpha} \eta \,dx \leq \int |u_x u| (1+x_+)^{2\alpha - 1/2} \eta \,dx  \\\leq 
\Bigl( \int_{x_0-2-b_+T}^{+\infty} |u_x|^2 (1+x_+)^{2\alpha-1} \,dx \int |u|^2 (1+x_+)^{2\alpha} \,dx \Bigr)^{1/2}\ \leq c(t_0, x_0),
\end{multline*}
$$
\int |d(x) u| (1+x_+)^{\alpha} \eta \,dx \leq \Bigl( \int |d(x)|^2 \,dx \int |u|^2 (1+x_+)^{2\alpha} \,dx \Bigr)^{1/2} \leq c(T).
$$
Therefore, since $\alpha-1/2 >1/4$, $\alpha \geq 1/4$, Theorem~\ref{T2.1} can be applied. In particular, for $t\in [t_0,T]$, $x\geq x_0$
\begin{multline}\label{5.11}
u_x(t,x) = -i \int_0^t \!\! \int G(t-\tau,x-y) f_1(\tau,y)\, dy d\tau \\ -i
\int_0^t\!\! \int G_x(t-\tau, x-y) f_2(\tau,y)\, dy d\tau.
\end{multline}
Note that if $y \in \supp \eta'$ and $x\geq x_0$ then $x-y \geq 1 + b_{+}T$. Then according to \eqref{A.27}
$$
|\partial_x^n G(t-\tau, x-y|) \leq \frac{c(T,n)}{t^{(n+1)/3}} e^{-c_0 t^{-1/2}} \leq c_1(T,n).
$$
Therefore,
\begin{multline*}
\Bigl| \int_0^t \!\! \int \partial^n_x G(t-\tau,x-y) u_{yy} \eta' \varphi \,dy d\tau \Bigr| \leq c(T,n) \int_{t_0/2}^T \int_{x_0-2-b_{+}T}^{x_0} |u_{yy}| \, dy d\tau  \\ \leq c_1(T,n,t_0,x_0).
\end{multline*}
It is easy to see that other terms in the right-hand side of \eqref{5.9} belong to $L_\infty(0,T;L_2^{\alpha-1/2})$ and can be estimated similarly to \eqref{5.4}, other terms in the right-hand side of \eqref{5.10} belong to $L_\infty(0,T;L_1^\alpha)$ and can be estimated similarly to \eqref{5.6}. As a result,
$$
\sup\limits_{(t,x) \in [t_0,T]\times [x_0,+\infty)} |u_x(t,x)| < \infty.
$$
Moreover the integrals in the right-hand side of \eqref{5.11} are continuous with respect to $(t,x)$. Taking into account Theorem \ref{T5.1}, we finish the proof.
\end{proof}

\section{Large-time decay}\label{S6}

The prehistory of the methods used in this section see, for example, in \cite{CCFSV, CDFN}.

\begin{lemma}\label{L6.1}
Let the function $d(x)\in H^1$ satisfies Condition~A for given $\alpha_0$ and $R_0$. Then for any $T>0$ and $M>0$ there exists a positive
constant $c_0=c_0(T,M,\alpha_0,R_0,\|d\|_{L_\infty},\beta)$ such that if $u_0\in H^1$ and $\|u_0\|_{L_2}\leq M$ then for a solution $u(t,x)$ 
to problem \eqref{1.1}, \eqref{1.2} from the class $X(\Pi_T)$ the following inequality holds:
\begin{equation}\label{6.1}
\int_{0}^{T}\!\!\int_{-R_0}^{R_0} u^2\,dxdt\leq c_0\iint_{\Pi_T} d(x)u^2\,dxdt.
\end{equation}
\end{lemma}

\begin{proof}
We argue by contradiction. Let us suppose that \eqref{6.1} is not verified and let $\{u_{0k}(x)\}_{k\in\mathbb{N}}$, $\{d_k(x)\}_{k\in\mathbb{N}}$, $\{\beta_k\}_{k\in \mathbb{N}}$ be sequences of
initial data from $H^1$ bounded in $L_2$, damping functions from $H^1$ verifying Condition A uniformly with respect to $k$ bounded in $L_{\infty}$  and uniformly bounded coefficients, where the corresponding solutions $\{u_k(t,x)\}_{k\in\mathbb{N}}$ verify
\begin{equation}\label{6.2}
\displaystyle\lim_{k \to+\infty}
\frac{\displaystyle \iint_{\Pi_T} d_k(x) u_k^2\,dxdt}{\displaystyle \int_0^T\!\!\int_{-R_0}^{R_0} u_k^2\,dxdt} =0.
\end{equation}
Define
$$
\nu_k= \|u_k\|_{L_2((0,T) \times (-R_0,R_0))},\qquad v_k(t,x)\equiv \frac{u_k(t,x)}{\nu_k}.
$$
Then
\begin{equation}\label{6.3}
\|v_k\|_{L_2((0,T) \times (-R_0,R_0))}=1 \qquad \forall k\in \mathbb N
\end{equation}
and $v_k$ is a solution from $X(\Pi_T)$ to a problem
\begin{equation}\label{6.4}
i v_{kt} +a v_{kxx} +i b v_{kx} +i v_{kxxx}+ \lambda\nu_k |v_k| v_k + i\beta_{k} \nu_k \bigl(|v_k| v_{k}\bigr)_x  +id_k(x)v_k =0,
\end{equation}
\begin{equation}\label{6.5}
v_k(0,x)=v_{0k}(x)\equiv \frac{u_{0k}(x)}{\nu_k}.
\end{equation}
Note that by virtue of \eqref{3.22}
\begin{equation}\label{6.6}
\nu_k\leq T^{1/2} M
\end{equation}
and it follows from \eqref{6.2} that
\begin{equation}\label{6.7}
\lim_{k\to +\infty}\iint_{\Pi_T} d_k(x) v_k^2\,dxdt=0.
\end{equation}
Next, we show that $\{v_{0k}(x)\}_{k\in\mathbb{N}}$ is bounded in $L_2$.
Indeed, properties of the functions $d_k(x)$ yield that
$$
\iint_{\Pi_T} u_k^2\,dxdt \leq
\frac 1{\alpha_0} \iint_{\Pi_T} d_k(x) u_k^2\,dxdt+ \int_0^T\!\!\int_{-R_0}^{R_0} u_k^2\,dxdt,
$$
so it succeeds from \eqref{3.22} that
\begin{multline*}
\int u_{0k}^2\,dx \leq \frac 1T \iint_{\Pi_T} u_k^2\,dxdt+
2\iint_{\Pi_T} d_k(x) u_k^2\,dxdt \\ \leq
\left(2+\frac 1{\alpha_0T}\right)\iint_{\Pi_T} d_k(x) u_k^2\,dxdt
+ \frac 1T \int_0^T\!\!\int_{-R_0}^{R_0} u_k^2\,dxdt
\end{multline*}
and, therefore,
$$
\int v_{0k}^2\,dxdt \leq
\left(2+\frac 1{\alpha_0 T}\right)\iint_{\Pi_T} d_k(x) v_k^2\,dxdt+\frac 1T,
$$
whence with the use of \eqref{6.7} the desired result on boundedness of
$\{v_{0k}\}_{k\in\mathbb{N}}$ follows.

Then according to \eqref{3.57} and \eqref{6.6} we deduce that
\begin{equation}\label{6.8}
\{v_{k}\}_{k\in\mathbb{N}} \quad \text{is bounded in}\quad L_\infty(0,T;L_2)\cap
L_2(0,T;H^1(I_n))\quad \forall n\in \mathbb N, \ I_n =(-n,n).
\end{equation}
Since $ L_1 (I_n)  \hookrightarrow H^{-1}(I_n)$ it follows from \eqref{6.8} that
$$
\{|v_{k}|v_{k}\}_{k\in\mathbb{N}} \quad \text{is bounded in}\quad
L_2(0,T;H^{-1}(I_n))\quad \forall n\in\mathbb N
$$
and then with the use of equation \eqref{6.4} itself, we obtain that
$$
\{v_{kt}\}_{k\in\mathbb{N}} \quad \text{is bounded in}\quad
L_2(0,T;H^{-2}(I_n))\quad \forall n\in\mathbb N.
$$
Passing to subsequences $\nu_k$, $\beta_{k}$, $\{v_{0k}\}$ and $\{v_{k}\}$,
still denoted by $\nu_k$, $\beta_{k}$, $\{v_{0k}\}$, $\{v_{k}\}$ from now on,
we get
\begin{gather}\notag
\nu_k \rightarrow \nu,\quad \beta_{k} \rightarrow \beta,\\
\notag
v_{0k}\rightharpoonup v_0\quad  \text{weakly in}\quad L_2,\\
\label{6.9}
v_k\rightharpoonup v\quad  *-\text{weakly in}\quad L_\infty(0,T;L_2),\\
\notag
v_k\rightharpoonup v\quad  \text{weakly in}\quad L_2(0,T;H^1(I_n))\quad \forall n\in\mathbb N
\end{gather}
and since $H^1(I_n) \overset{c}\hookrightarrow L_2(I_n) \hookrightarrow H^{-2}(I_n)$ by the standard argument that
$$
v_k\rightarrow v \quad \mbox{strongly in}\ \ L_2\bigl((0,T) \times I_n \bigr)\quad \forall n\in\mathbb N.
$$
Note that it follows from \eqref{6.7} and Condition A that $v_k\to 0$ in
$L_2\bigl((0,T) \times (\mathbb R\setminus (-R_0,R_0))\bigr)$, therefore,
\begin{equation}\label{6.10}
v_k\rightarrow v \quad \text{strongly in}\quad L_{2}(\Pi_T),
\end{equation}
where
\begin{equation}\label{6.11}
v(t,x)=0\quad\text{for}\ |x|>R_0.
\end{equation}

Let $\phi(t,x)$ be an arbitrary function such that $\phi\in C^1([0,T];H^1) \cap C([0,T];H^3)$ and
$\phi\bigr|_{t=T}=0$. For any $k$
\begin{multline*}
\iint_{\Pi_T} \bigl[ i v_k\phi_t  -a v_k \phi_{xx} +  i b v_k\varphi_x+ i v_k \phi_{xxx} -\lambda \nu_k |v_k| v_k \phi + i \beta_{k} \nu_k |v_k| v_k \phi_x  \\ -i d_k(x) v_k\phi\bigr]\,dxdt +\int v_{0k}\phi\bigr|_{t=0}\,dx=0.
\end{multline*}
Passing to the limit as $k\to +\infty$ and taking into account \eqref{6.7} and boundedness of $d_k$ in $L_\infty$ we obtain an equality
\begin{multline*}
\iint_{\Pi_T} \bigl[ i v\phi_t -a v \phi_{xx} + i b v\varphi_x  + i v \phi_{xxx}- \lambda \nu |v| v \phi  + i \beta \nu |v| v \phi_x \bigr]\,dxdt  \\+\int v_{0}\phi\bigr|_{t=0}\,dx=0.
\end{multline*}
that is, $v(t,x)$ is a weak solution to a problem
\begin{equation}\label{6.12}
i v_{t} +a_1 v_{xx}  + i b v_{x} + i v_{xxx}+ \lambda\nu |v| v + i \beta \nu \bigl(|v| v\bigr)_x =0,\quad
v(0,x)=v_{0}(x).
\end{equation}
It follows from \eqref{6.9} and \eqref{6.11} that $v\in L_\infty(0,T; L_2^{3/4})$, that is, it lies in the class of uniqueness. Then it follows from Theorem~\ref{T5.2} that $u\in L_2(t_0,T; H^2(I))$ and from Theorem~\ref{T5.3} that $u, u_x \in L_\infty((t_0,T)\times I)$ for any $t_0\in (0,T)$ and bounded interval $I\subset \mathbb R$. Note that according to \eqref{B.1}
$$
\bigl( |u| u\bigr)_x =  \frac 32 |u| u_x + \frac {u^2 \bar u_x}{2 |u|} = \frac12 P(u,u_x),
$$
where the function $P$ is defined by formula \eqref{B.8}. Then according to \eqref{B.9} equation \eqref{6.12} satisfies the hypothesis of Theorem~6 from \cite{Sh} and the unique continuation property from \cite{Sh} provides that $v(t,x) \equiv 0$ in $\Pi_T$. Thus, \eqref{6.10} yields
that $v_k\to 0$ in $L_2(\Pi_T)$, which contradicts \eqref{6.3}.
\end{proof}

\begin{theorem}\label{T6.1}
Let the function $d(x)\in H^1$ satisfies Condition~A, $\beta \ne 0$. Then there exists a constant $\gamma>0$ and for any $M>0$ there exists a constant $c(M)>0$ such that if $u_0 \in H^1$ and $\|u_0\|_{L_2} \leq M$ the corresponding function $u\in X(\Pi_T)$ $\forall T>0$ which is the solution to problem \eqref{1.1}, \eqref{1.2} in $\Pi_T$, constructed in Theorem \ref{T3.2} for $\alpha=0$, satisfies inequality \eqref{1.5}.
\end{theorem}

\begin{proof}
Let $T>0$ be fixed. According to \eqref{3.22} $\|u(t,\cdot)\|_{L_2}$ does not increase with respect to $t$ and
\begin{equation}\label{6.13}
\int |u(T,x)|^2 \,dx + 2\iint_{\Pi_T} d(x) |u(t,x)|^2 \,dxdt = \int |u_0|^2 \,dx.
\end{equation}
The properties of the function $d(x)$ provide that
$$
\int |u(T,x)|^2\,dx+ 2\alpha_0\int_0^T\!\!\int_{\mathbb R \setminus (-R_0,R_0)} |u|^2\,dxdt \leq \int |u_0|^2\,dx
$$
and, consequently,
$$
\iint_{\Pi_T} |u|^2\,dxdt \leq \frac1{2\alpha_0} \int |u_0|^2\,dx - \frac1{2\alpha_0} \int |u(T,x)|^2 \,dx+ \int_0^T\!\!\int_{-R_0}^{R_0} |u|^2\,dxdt.
$$
By virtue of \eqref{6.1} if $\|u_0\|_{L_2} \leq M$, for $c_0= c_0(M)$ the last inequality can be transformed to the following one:
$$
\iint_{\Pi_T} |u|^2\,dxdt \leq \frac1{2\alpha_0} \int |u_0|^2\,dx- \frac1{2\alpha_0} \int |u(T,x)|^2 \,dx+ c_0\iint_{\Pi_T} d(x)|u|^2\,dxdt,
$$
which with the use of \eqref{6.13} and monotonicity of $\|u(t,\cdot)\|_{L_2}$ yields that
\begin{multline*}
T\int |u^2(T,x)|^2 \,dx \leq \frac1{2\alpha_0} \int  |u_0|^2 \,dx- \frac1{2\alpha_0} \int |u(T,x)|^2 \,dx \\ +
\frac{c_0}2 \Bigl(\int |u_0|^2\,dx- \int |u(T,x)|^2 \,dx\Bigr),
\end{multline*}
whence it follows that
$$
\left(T+\frac1{2\alpha_0}+\frac{c_0}2\right) \int |u(T,x)|^2 \,dx \leq \left(\frac1{2\alpha_0}+\frac{c_0}2\right) \int |u_0|^2\,dx,
$$
which is equivalent to an inequality
\begin{equation}\label{6.14}
\int |u(T,x)|^2 \,dx \leq \frac{1+ c_0\alpha_0}{1+ c_0\alpha_0 +2T\alpha_0} \int |u_0|^2 \,dx, 
\end{equation}
whence with the use of the semigroup property follows that 
\begin{equation}\label{6.15}
\|u(t,\cdot)\|_{L_2}^2 \leq c(M) e^{-\gamma(M) t} \quad \forall\ t\geq 0,
\end{equation}
where
\begin{equation}\label{6.16}
\gamma(M) = -\frac 1T \ln\frac{1+ c_0\alpha_0}{1+ c_0\alpha_0 +2T\alpha_0}>0,\quad c(M) = \left(1 + \frac{2T\alpha_0}{1 + c_0\alpha_0}\right) M^2.
\end{equation}

Define $\gamma = \gamma(1)$, that is the constant from \eqref{6.16} corresponding to $M=1$. 

For any $M>0$ according to \eqref{6.15} there exists $T_0 = T_0(M)>0$ such that $\|u(T_0,\cdot)\|_{L_2} \leq 1$ if $\|u_0\|_{L_2} \leq M$. Then for $t\geq T_0$
$$
\|u(t,\cdot)\|^2_{L_2} \leq c(1) e^{\gamma T_0} e^{-\gamma t},
$$
while for $t\in [0,T_0)$
$$
\|u(t,\cdot)\|^2_{L_2} \leq M^2 e^{\gamma T_0} e^{-\gamma t},
$$
whence \eqref{1.5} evidently follows.
\end{proof}

\begin{theorem}\label{T6.2}
Let the function $d(x)\in L_\infty$ satisfies Condition~A. Then there exists a constant $\gamma>0$ and for any $M>0$ there exists a constant $c(M)>0$ such that if $u_0 \in L_2$ and $\|u_0\|_{L_2} \leq M$ the corresponding function $u\in C_w([0,T];L_2)$ $\forall T>0$ which is the solution to problem \eqref{1.1}, \eqref{1.2} in $\Pi_T$, constructed in Theorem~\ref{T3.3} for $\alpha=0$, satisfies inequality \eqref{1.5}. 
\end{theorem}

\begin{proof}
Approximate the functions $d$ and $u_0$ by smooth functions $d_h$ and $u_{0h}$ and introduce the functions $\beta_{h}$ as it is done in the proof of Theorem~\ref{T3.3} and consider the corresponding solutions $u_h(t,x)\in X(\Pi_T)$ $\forall T>0$. Note that if $h\in (0,1]$ Condition A is verified for functions $d_h$ uniformly with respect to $h$. Moreover, the functions $u_{0h}$ are bounded in $L_2$, the functions $d_h$ are bounded in $L_\infty$ and the values $\beta_{h}$ are bounded also uniformly with respect to $h$. Then inequality \eqref{6.1} is also uniform and the simple revision of the proof of Theorem \ref{T6.1} shows that functions $u_h$ verify estimate \eqref{1.5} uniformly with respect to $h$. 

In the proof of Theorem~\ref{3.3} the solution $u$ is constructed, in particular, as the $*$-weak limits of functions $u_h$ in $L_\infty(0,T;L_2)$ for any $T>0$. Then 
for any $t\in [T,T+1]$
$$
\|u(t,\cdot)\|^2_{L_2} \leq c(M) e^{-\gamma T} \leq c(M) e^{\gamma} e^{-\gamma t},
$$
which finishes the proof.
\end{proof}

Finally, we can present the proof of the main theorem.

\begin{proof}[Proof of Theorem~\ref{T1.1}]
Existence and uniqueness of the weak solution $u\in C_w([0,T];L_2^{3/4})$ succeeds from Theorems~\ref{T3.3} and~\ref{T4.1}, while estimate \eqref{1.5} under Condition~A --- from Theorem~ref{T6.2}.
\end{proof}

\appendix
\section{}\label{A}

Let $A>0$ and $a\in \mathbb R$, $|a|\leq A$. Introduce a function
\begin{equation}\label{A.1}
\Phi_a(x) \equiv \frac1{2\pi} \int e^{i(\xi^3 -a\xi^2 +x\xi)}\,d\xi.
\end{equation}

For this function we establish properties similar to the well-known ones for the classical Airy function.

\begin{lemma}\label{LA.1}
The function $\Phi_a$ for any $x\in \mathbb R$ exists, is infinitely smooth and satisfies an equation
\begin{equation}\label{A.2}
3 y'' -2a i y' - xy =0.
\end{equation} 
Moreover, for any integer non-negative $n$
\begin{equation}\label{A.3}
\bigl| \Phi_a^{(n)}(x)\bigr| \leq
\begin{cases} \displaystyle c(A,n) \bigl(1+|x|\bigr)^{(2n-1)/4},\quad &x\leq 0,\\ 
c(A,n) e^{-c_0 x^{3/2}},\quad &x\geq 0,
\end{cases}
\end{equation}
for certain positive constants $c_0$, $c(A,n)$. Finally, $\Phi_a = \EuScript F^{-1}\left[e^{i(\xi^3 -a\xi^2 )}\right]$, where the inverse Fourier transform is understood as an operation in $\EuScript S'$.
\end{lemma}

\begin{proof}
Let $n=0$ or $n=1$. Consider functions of the complex variable $z$
\begin{equation}\label{A.4}
\varphi (z) \equiv z^3 - a z^2 +x z,\quad g_n(z) \equiv z^n e^{i\varphi (z)}.
\end{equation}
For $R>0$ define a contour $\gamma$ on $\mathbb C$, consisting of the segments $[0,R]$, $[Re^{i\pi/6},0]$ and the  connecting them arc of a circumference $\gamma_R(0,\pi/6)$ of the radius $R$. Then $\displaystyle \int_\gamma g_n(z)\, dz =0$. After the change of variables $z = \zeta^{1/3}$ (the choice of the branch of the root is evident) one obtains that
\begin{equation}\label{A.5}
\int_{\gamma_R(0,\pi/6)} g_n(z)\, dz = \frac13 \int_{\gamma_{R^3} (0,\pi/2)} \zeta^{(n-2)/3} e^{i(\zeta - a \zeta^{2/3} +x \zeta^{1/3})}\, d\zeta.
\end{equation}
Define on the arc $\gamma_{R^3}(0,\pi/2)$ the following parameter: let $\zeta = R^3 e^{it}$, $t\in [0,\pi/2]$. Then $\displaystyle \frac2{\pi}t \leq \sin t \leq t$ and, therefore,
$$
\left| e^{i(\zeta - a \zeta^{2/3} +x \zeta^{1/3})}\right | = e^{-R^3 \sin t + a R^2 \sin (2t/3) - x R \sin (t/3)} \leq e^{- R^3 t /\pi}
$$
for considerably large $R$ uniformly on any compact set with respect to $x$. Since $n < 2$ it follows from \eqref{A.5} that 
\begin{equation}\label{A.6}
\int_{\gamma_R(0,\pi/6)} g_n(z)\, dz \to 0
\end{equation}
when $R\to +\infty$ uniformly on any compact set with respect to $x$. Next, after the change of variables $z = r e^{i\pi/6}$ we find that
\begin{equation}\label{A.7}
\int_0^{Re^{i\pi/6}} g_n(z)\,dz = e^{i\pi(n+1)/6} \int_0^R r^n e^{- r^3 + a r^2 \sqrt{3}/2 -x r/2} e^{i(-a r^2/2 +x r \sqrt{3}/2)}\,dr.
\end{equation}
Therefore, when $R\to +\infty$ uniformly on any compact set with respect to $x$
\begin{equation}\label{A.8}
\int_0^{Re^{i\pi/6}} g_n(z)\,dz \to e^{i\pi(n+1)/6} \int_0^{+\infty} r^n e^{- r^3 + a r^2 \sqrt{3}/2 -x r/2} e^{i(-a r^2/2 +x r \sqrt{3}/2)}\,dr.
\end{equation}
As a result, it follows from \eqref{A.6} and \eqref{A.7} that when $R\to +\infty$ uniformly on any compact set with respect to $x$
\begin{multline}\label{A.9}
\int_0^R g_n(\xi)\, d\xi \to \int_0^{+\infty} g_n(\xi)\, d\xi = \int _0^{+\infty\cdot e^{i\pi/6}} g_n(z) \, dz \\ =
e^{i\pi(n+1)/6} \int_0^{+\infty} r^n e^{-r^3 + a r^2 \sqrt{3}/2 -x r/2} e^{i(-a r^2/2 +x r \sqrt{3}/2)}\,dr.
\end{multline}

Similarly, choosing a contour consisting of the segments $[-R,0]$, $[0,Re^{5i\pi/6}]$ and the arc $\gamma_R(5\pi/6,\pi)$ we find that when $R\to +\infty$ uniformly on any compact set with respect to $x$
\begin{multline}\label{A.10}
\int_{-R}^0 g_n(\xi)\, d\xi \to \int_{-\infty}^0 g_n(\xi)\, d\xi = \int _{+\infty\cdot e^{5i\pi/6}}^0 g_n(z) \, dz \\ =
-e^{5i\pi(n+1)/6} \int_0^{+\infty} r^n e^{-r^3 - a r^2 \sqrt{3}/2 -x r/2} e^{i(-a r^2/2 -x r \sqrt{3}/2)}\,dr.
\end{multline}

As a result, it follows from \eqref{A.9} and \eqref{A.10} that the function $\displaystyle \Phi_a(x) = \frac1{2\pi} \int_{\mathbb R} g_0(\xi)\, d\xi$ exists and
\begin{multline}\label{A.11}
\Phi_a(x) = \frac1{2\pi} \Bigl[ \int_0^{+\infty} e^{- r^3 + a r^2 \sqrt{3}/2 -x r/2} e^{i(-a r^2/2 +x r \sqrt{3}/2+ \pi/6)}\, dr \\ +
\int_0^{+\infty} e^{-r^3 - a r^2 \sqrt{3}/2 -x r/2} e^{i(-a r^2/2 -x r \sqrt{3}/2- \pi/6)}\, dr \Bigr].
\end{multline}
In particular, this function is infinitely smooth. Differentiating equalities \eqref{A.9} and \eqref{A.10} for $n=0$ with respect to $x$ we derive that
\begin{multline}\label{A.12}
\Phi_a'(x) = \frac{i}{2\pi} \Bigl[ \int_{+\infty\cdot e^{5i\pi/6}}^0 g_1(z)\, dz + \int_0^{+\infty\cdot e^{i\pi/6}} g_1(z)\, dz\Bigr] \\ =
\frac{i}{2\pi} \int g_1(z)\, dz = \frac{i}{2\pi} \int \xi e^{i(\xi^3 -a\xi^2 +x\xi)}\,d\xi.
\end{multline}
Moreover, since $\varphi'(z) = 3 z^2 -2a z +x$
\begin{multline*}
3 \Phi_a''(x) -2a i \Phi_a'(x) -x \Phi_a(x) \\ = 
\frac1{2\pi} \Bigl(  \int_{+\infty\cdot e^{5i\pi/6}}^0 + \int_0^{+\infty\cdot e^{i\pi/6}} \Bigr) (-3 z^2 +2a z -x) e^{i\varphi(z)}\, dz \\ =
\frac{i}{2\pi} e^{i\varphi(z)}\Big|_{+\infty\cdot e^{5i\pi/6}}^{+\infty\cdot e^{i\pi/6}} =0.
\end{multline*}
Note that $\displaystyle \frac1{2\pi} \int_{-R}^R  e^{i(\xi^3 -a\xi^2 +x\xi)}\,d\xi \to \Phi_a(x)$ when $x\to +\infty$ uniformly on any compact set with respect to $x$. Therefore, this convergence is valid in the space $\EuScript D'(\mathbb R)$. On the other hand, $e^{i(\xi^3 -a\xi^2)}\chi_{(-R,R)}(\xi) \to 
e^{i(\xi^3 -a\xi^2)}$ in the space $\EuScript S'(\mathbb R)$ when $R\to +\infty$. Since the inverse Fourier transform is continuous in this space, it follows that
$\displaystyle \frac1{2\pi} \int_{-R}^R  e^{i(\xi^3 -a\xi^2 +x\xi)}\,d\xi \to \EuScript F^{-1} \bigl[e^{i(\xi^3 -a\xi^2)}\bigr](x)$ in $\EuScript S'$. As a result, $\Phi_a = 
\EuScript F^{-1} \bigl[e^{i(\xi^3 -a\xi^2)}\bigr]$ in $\EuScript S'$.

It follows from \eqref{A.11} that for any integer non-negative $n$
\begin{equation}\label{A.13}
\bigl| \Phi_a^{(n)}(x) \Bigr| \leq c(A,n),\quad |x|\leq 1.
\end{equation}

In order to obtain estimates \eqref{1.3} for $|x|>1$ first consider the case $n\leq 1$. Then according to \eqref{A.1} and \eqref{A.12} we have to estimate $\displaystyle \int_{\mathbb R} g_n(\xi)\, d\xi$.

Let $x>1$. For $R>0$ construct on a complex plane a rectangle $\widetilde\gamma$ with the base $[-R,R]$ and the height $h = (x/3)^{1/2}$. Then
$\displaystyle \int_{\widetilde\gamma} g_n(z)\, dz =0$. Consider first the right side of the rectangle. Introduce a parameter $z = R +it$, $t\in [0,\sqrt{x/3}]$. Then
\begin{multline*}
\int_R^{R+i\sqrt{x/3}} g_n(z) \,dz \\ =
i\int_0^{\sqrt{x/3}} (R+it)^n e^{-3R^2t +t^3 +2a Rt -xt} e^{i(R^3 -3Rt^2 -aR^2 +at^2 +xR)}\,dt.
\end{multline*}
For the fixed $x$ and considerably large $R$
\begin{multline*}
\Bigl| \int_R^{R+i\sqrt{x/3}} g_n(z) \,dz \Bigr| \leq \left(R^2 +\frac{x}{3}\right)^{n/2} \int_0^{\sqrt{x/3}} e^{-R^2t}\, dt \\ \leq
\frac1{R^2}  \left(R^2 +\frac{x}{3}\right)^{n/2} \int_0^{+\infty} e^{-\theta}\, d\theta \to 0,\quad R\to +\infty.
\end{multline*}
The left side of the rectangle is considered in a similar way. Then passing to the limit when $R\to +\infty$ we obtain that
\begin{multline*}
\int g_n(\xi)\, d\xi = \int_{-\infty +i\sqrt{x/3}}^{+\infty +i\sqrt{x/3}} g_n(z) \,dz \\ =
\int \bigl(\xi + i\sqrt{x/3}\bigr)^n e^{i[(\xi + i\sqrt{x/3})^3 -a(\xi + i\sqrt{x/3})^2 +x(\xi + i\sqrt{x/3})]}\,d\xi \\ =
e^{-2(x/3)^{3/2}} \int \bigl(\xi + i\sqrt{x/3}\bigr)^n e^{-\sqrt{3x}\xi^2 +2a\sqrt{x/3}\xi} e^{i\Re\varphi(\xi+ i\sqrt{x/3})}\,d\xi,
\end{multline*}
whence it follows that
\begin{multline*}
\Bigl| \int g_n(\xi)\, d\xi \Bigr| \leq e^{-2(x/3)^{3/2}} \int \bigl(\xi^2 +\frac{x}{3}\bigr)^{n/2} e^{-\sqrt{3x}\xi^2 +2a\sqrt{x/3}\xi}\, d\xi \\ = e^{-2(x/3)^{3/2} +a^2 x^{1/2} 3^{-3/2}} \int  \bigl(\xi^2 +\frac{x}{3}\bigr)^{n/2}  e^{-\sqrt{3x}(\xi- a/3)^2}\, d\xi \\ \leq
c(A,n) e^{-(x/3)^{3/2}}
\end{multline*}
and estimate \eqref{A.3} in the case $x>1$, $n\leq 1$ is established.

Now let $x<-1$, $n\leq 1$. 
Divide the real line into three parts:
\begin{gather*}
\Omega_1 = \{\xi: |x|/2 \leq 3 \xi^2 - 2a\xi  \leq 3|x|/2\},\\ \Omega_2 = \{\xi: 3 \xi^2 - 2a\xi < |x|/2\}, \\ \Omega_3 = \{\xi: 3\xi^2 -2a\xi > 3|x|/2\}. 
\end{gather*}
Let $\xi \in \Omega_1$. Consider first the inequality $3\xi^2 -2a\xi\geq |x|/2$. Let
\begin{equation}\label{A.14}
\xi_1 = \frac{1}{3}\bigl(a - \sqrt{a^2 +3 |x|/2}\bigr),\quad \xi_2 = \frac{1}{3} \bigl(a + \sqrt{a^2 +3 |x|/2}\bigr). 
\end{equation}
It means that either $\xi\leq \xi_1$ or $\xi\geq \xi_2$. If $\xi\geq \xi_2$ then $\xi \geq |x/6|^{1/2}$ in the case $a\geq 0$, while in the case $a<0$
$$
\xi \geq \frac{1}{3} \bigl( \sqrt{A^2 +3|x|/2} - A\bigr) \geq \frac{1}{3} \bigl( \sqrt{A^2 +3/2} -A\bigr) |x|^{1/2}.
$$ 
If $\xi\leq \xi_1$ then $|\xi| \geq |x/6|^{1/2}$ in the case $a\leq 0$, while in the case $a>0$
$$
|\xi| \geq \frac{1}{3} \bigl( \sqrt{A^2 +3|x|/2} - A\bigr) \geq \frac{1}{3} \bigl( \sqrt{A^2 +3/2} -A\bigr) |x|^{1/2}.
$$ 
As a result, if $\xi \in \Omega_1$ then $|\xi| \geq c^{-1}(A) |x|^{1/2}$. Next, consider the inequality $3 \xi^2 - 2a\xi \leq 3|x|/2$. Let
\begin{equation}\label{A.15}
\xi_3 = \frac{1}{3}\bigl(a - \sqrt{a^2 +9 |x|/2}\bigr),\quad \xi_4 = \frac{1}{3} \bigl(a + \sqrt{a^2 +9 |x|/2}\bigr). 
\end{equation}
It means that $\xi_3 \leq \xi \leq \xi_4$. Then if $a\geq 0$
$$
|\xi|\leq \xi_4 \leq \frac{1}{3} \bigl( \sqrt{A^2 +9/2} + A\bigr)|x|^{1/2},
$$ 
while if $a<0$
$$
|\xi| \leq |\xi_3|  \leq \frac{1}{3} \bigl( \sqrt{A^2 +9/2} + A\bigr)|x|^{1/2}.
$$
As a result, if $\xi\in \Omega_1$ then
\begin{equation}\label{A.16}
c^{-1}(A) |x|^{1/2} \leq |\xi| \leq c(A) |x|^{1/2}.
\end{equation}
Note that $\varphi''(\xi) = 6\xi  - 2a$. Then for $\xi\in \Omega_1$
\begin{equation}\label{A.17}
|\varphi''(\xi)| = \frac1{|\xi|} |3 \xi^2 -2a \xi + 3 \xi^2| \geq  3 |\xi| \geq c^{-1}(A)|x|^{1/2}
\end{equation}
and application of the Van-der-Corput lemma yields that
\begin{multline}\label{A.18}
\Bigl| \int_{\Omega_1} g_n(\xi)\,d\xi\Bigr| \\ \leq
c\Bigl(\inf\limits_{\xi\in\Omega_1} |\varphi''(\xi)|\Bigr)^{-1/2} \Bigl(\sup\limits_{\xi\in\Omega_1} |\xi|^n +\int_{\Omega_1} \bigl|(\xi^n)'\bigr|\,d\xi \Bigr) \leq
c(A)|x|^{(2n-1)/4}.
\end{multline}

Let $\xi\in \Omega_2$. Then $\xi_1 \leq \xi \leq \xi_2$. It means that
$$
|\xi| \leq \frac{1}{3} \bigl( \sqrt{A^2 +3/2} + A\bigr)|x|^{1/2} = c(A)|x|^{1/2}.
$$
Therefore, 
\begin{equation}\label{A.19}
|\varphi'(\xi)| = |3\xi^2 -2a\xi +x| \geq |x|/2 \geq c^{-1}(A)(1+\xi^2).
\end{equation}
Then
\begin{equation}\label{A.20}
\int_{\Omega_2} g_n(\xi)\,d\xi = \frac{\xi^n}{i\varphi'(\xi)} e^{i\varphi(\xi)}\Big|_{\xi_1}^{\xi_2} 
-\frac1i \int_{\Omega_2} \Bigl(\frac{\xi^n}{\varphi'(\xi)}\Bigr)' e^{i\varphi(\xi)}\,d\xi.
\end{equation}
Here
$$
\Bigl(\frac{\xi^n}{\varphi'(\xi)}\Bigr)' = \frac{(\xi^n)'\varphi'(\xi) - \xi^n \varphi''(\xi)}{(\varphi'(\xi))^2}
$$
and since $|\xi^n\varphi''(\xi)| \leq c(A)(\xi^2 +1)$
\begin{multline}\label{A.21}
\Bigl|\int_{\Omega_2} g_n(\xi)\,d\xi\Bigr| \leq c(A)|x|^{n/2-1} + c(A)\int_{\Omega_2} \frac{d\xi}{1+\xi^2}  \leq c(A)|x|^{n/2-1} \\ 
+ \frac{c_1(A)}{|x|^{1/4}} \int \frac{d\xi}{(1+\xi^2)^{3/4}} \leq c_2(A)|x|^{-1/4} \leq c_2(A) |x|^{(2n-1)/4}.
\end{multline}

Let $\xi\in \Omega_3$. Then either $\xi\leq \xi_3$ or $\xi\geq \xi_4$ and 
\begin{equation}\label{A.22}
|\varphi'(\xi)| \geq |3\xi^2 -2a\xi| -|x| \geq |x|/2 \geq 1/2.
\end{equation}
Since $|x| \leq 2(3 \xi^2 - 2a\xi)/3$ it follows that
$$
|\varphi'(\xi)| \geq 3\xi^2 - 2a\xi -|x| \geq \frac13 (3\xi^2 - 2a\xi).
$$
If $a\xi \leq 0$, then with the use of \eqref{A.22} we obtain that
\begin{equation}\label{A.23}
|\varphi'(\xi)|  \geq (1+ \xi^2)/4.
\end{equation}
Let $a\xi >0$. If $a >0$ then $\xi\geq \xi_4$, if $a<0$ then $\xi \leq \xi_3$. In both cases we obtain that
$$
|\xi| \geq \frac1{3} \bigl( |a| + \sqrt{a^2 +9|x|/2}\bigr).
$$
Since $|x| \geq a^2/A^2$
$$
|\xi| \geq \frac{|a|}{3} \bigl( 1 +\sqrt{1+ \frac9{2A^2}}\bigr) = \frac2{1-\widetilde c} \frac{|a|}{3},
$$
where 
$$
\widetilde c =  \frac{\sqrt{A^2 +9/2} -A}{\sqrt{A^2+9/2} +A} = c^{-1}(A), \quad \widetilde c \in (0,1),
$$ 
that is $(1- \widetilde c) |\xi| \geq 2|a|/3$, and, therefore,
\begin{equation}\label{A.24}
|\varphi'(\xi)| \geq \widetilde c \xi^2 + |\xi| \bigl( (1-\widetilde c)|\xi| -2|a|/3\bigr) \geq \widetilde c  \xi^2 \geq c^{-1}(A)(1+\xi^2).
\end{equation}
Inequalities \eqref{A.22}, \eqref{A.23}, \eqref{A.24} mean that in the case $\xi\in \Omega_3$ the function $|\varphi'(\xi)|$ verifies the same properties as in the case $\xi\in \Omega_2$ in \eqref{A.19}.
Then we can apply the same argument as in \eqref{A.20} and obtain the same estimate as in \eqref{A.21}.

As a result, estimate \eqref{A.3} is established for $n\leq 1$. For other values of $n$ it follows from equality \eqref{A.2}.
\end{proof}

Now consider the fundamental solution of the differential operator 
$\EuScript L(\partial_t, \partial_x) = \partial_t +\partial_x^3 - i a \partial_x^2 + b\partial_x$ (see \eqref{2.5}).
Then it follows from \eqref{2.6} and \eqref{A.1} that
\begin{equation}\label{A.25}
G(t,x) = \frac{\theta(t)}{t^{1/3}} \Phi_{a t^{1/3}}\left(\frac{x- b t}{t^{1/3}}\right).
\end{equation}
The properties of the function $\Phi_a$ imply that the function $G(t,x)$ is infinitely smooth for $t>0$.

\begin{lemma}\label{LA.2}
For any $T>0$, $0<t\leq T$ and integer non-negative $n$ 
\begin{equation}\label{A.26}
\bigl| \partial_x^n G(t,x)\bigr| \leq
\begin{cases} \displaystyle c(T,n,a,b) t^{-q(n)}\bigl(1+|x|\bigr)^{(2n-1)/4},\quad &x\leq 0,\\ 
c(T,n,a,b) t^{-(n+1)/3}e^{-c_0 x^{3/2}T^{-1/2}},\quad &x\geq 0,
\end{cases}
\end{equation}
for certain positive constants $c(T,n,a,b)$, $c_0$, where $q(0) = 1/3$, $q(n) = (2n+1)/4$ for $n\geq 1$. Moreover, if $b\leq 0$, $x\geq 0$ or $b\geq 0$, $x\geq b t$, then
\begin{equation}\label{A.27}
\bigl| \partial_x^n G(t,x)\bigr| \leq c(T,n,a,b) t^{-(n+1)/3} e^{-c_0 x^{3/2}t^{-1/2}}.
\end{equation}
\end{lemma}
 
\begin{proof}
First consider the case $b\geq 0$. Let $x\geq b t$, then since
$$
\Bigl(\frac{x-b t}{t^{1/3}}\Bigr)^{3/2} \geq \frac1{\sqrt2} \Bigl(\frac{x}{t^{1/3}}\Bigr)^{3/2} - b^{3/2}T
$$
inequality \eqref{A.3} in the case $x\geq 0$ yields that
$$
|\partial_x^n G(t,x)| \leq \frac{c(|a|T^{1/3},n)}{t^{(n+1)/3}} e^{-c_0\bigl(\frac{x-b t}{t^{1/3}}\bigr)^{3/2}} \leq \frac{c(T,n,a,b)}{t^{(n+1)/3}} e^{-c_0 x^{3/2}(2t)^{-1/2}}.
$$
Let $0 \leq x \leq  b t$, then since 
$$
1\leq 1 +\frac{|x- b t|}{t^{1/3}} \leq 1 + b T^{2/3}
$$
inequality \eqref{A.3} in the case $x\leq 0$ yields that
\begin{multline*}
|\partial_x^n G(t,x)| \leq \frac{c(|a|T^{1/3},n,b T^{2/3})}{t^{(n+1)/3}} \leq \frac{c(T,n,a,b)}{t^{(n+1)/3}} e^{-c_0 (bT)^{3/2} T^{-1/2}} \\ \leq \frac{c(T,n,a,b)}{t^{(n+1)/3}} e^{-c_0 x^{3/2} T^{-1/2}}.
\end{multline*}
Finally, let $x\leq 0$, then since
$$
1+  \frac{|x|}{T^{1/3}} \leq 1 +\frac{|x- b t|}{t^{1/3}} \leq 1 + b T^{2/3} + \frac{|x|}{t^{1/3}}
$$
inequality \eqref{A.3} in the case $x\leq 0$ yields that
\begin{equation}\label{A.28}
|G(t,x)| \leq c(T,a,b) t^{-1/3} (1+|x|)^{-1/4}
\end{equation}
and if $n\geq 1$
\begin{equation}\label{A.29}
|\partial_x^n G(t,x)| \leq \frac{c(T,n,a,b)}{t^{(n+1)/3}} (1 +\frac{|x|}{t^{1/3}})^{(2n-1)/4} \leq c_1(T,n,a,b) t^{-q(n)} (1+|x|)^{(2n-1)/4}.
\end{equation}

Now consider the case $b<0$. If $x\geq 0$ then inequality \eqref{A.3} in the case $x\geq 0$ yields that
$$
|\partial_x^n G(t,x)| \leq \frac{c(|a|T^{1/3},n)}{t^{(n+1)/3}} e^{-c_0\bigl(\frac{x- bt}{t^{1/3}}\bigr)^{3/2}} \leq \frac{c(T,n,a)}{t^{(n+1)/3}} e^{-c_0 x^{3/2} t^{-1/2}}.
$$
If $b t \leq x \leq 0$, then $|x| \leq |b|T$ and inequality \eqref{A.3} in the case $x\geq 0$ yields that
\begin{multline*}
|\partial_x^n G(t,x)| \leq \frac{c(|a|T^{1/3},n)}{t^{(n+1)/3}} \leq \frac{c(T,n,a,b)}{t^{(n+1)/3}} (1+|x|)^{(2n-1)/4}\\ \leq \frac{c_1(T,n,a,b)}{t^{q(n)}} (1+|x|)^{(2n-1)/4}.
\end{multline*}
Finally, if $x\leq b t$ then since
$$
\frac{1}{1+|b|T^{2/3}} \Bigl(1+\frac{|x|}{T^{1/3}} \Bigr) \leq 1+ \frac{|x-b t|}{t^{1/3}}  \leq 1 + \frac{|x|}{t^{1/3}}
$$
inequality \eqref{A.3} in the case $x\leq 0$ similarly to \eqref{A.28}, \eqref{A.29} yields estimates \eqref{A.26} also in this case. 
\end{proof}

\begin{lemma}\label{LA.3}
For any $\alpha\geq 0$, $\epsilon>0$, $T>0$, $t\in (0,T]$, $n=0$ or $n=1$, and $y\in \mathbb R$
\begin{equation}\label{A.30}
\|\partial_x^n G(t,\cdot -y)\rho^{1/2}_{\alpha,\epsilon}\|_{L_2(\mathbb R)} \leq c(\alpha,\epsilon,T,a,b) t^{-(5n+4)/12} \rho^{1/2}_{\alpha,\epsilon}(y) (1+y_+)^{(2n+1)/4}.
\end{equation}
\end{lemma}

\begin{proof}
Note that $(5n+4)/12 = q(n)$ for $n=0$ and $n=1$, $q(n) \geq (n+1)/3$, where $q(n)$ is a number from \eqref{A.26}.

First let $y\geq 0$, then ($\rho\equiv \rho_{\alpha,\epsilon)}$)
$$
\int \rho(x) |\partial_x^n G(t,x-y)|^2\, dx = \int_{-\infty}^0\dots dx + \int _0^y\dots dx +\int_y^{+\infty}\dots dx.
$$ 
Inequalities \eqref{A.26} imply that
\begin{multline*}
t^{2q(n)} \int_{-\infty}^0 \dots dx \leq c \int_{-\infty}^0 (1+y-x)^{(2n-1)/2} e^{2\epsilon x}\,dx \\  \leq
c(1+y)^{1/2} \int_{-\infty}^0 (1+|x|)^{1/2} e^{2\epsilon x}\, dx   = c(\epsilon) (1+y)^{1/2},
\end{multline*}
\begin{multline*}
t^{2q(n)} \int_0^y \dots dx \leq c \int_0^y (1+y-x)^{(2n-1)/2} \rho(x) \,dx  \\ \leq 
c \rho(y) \int_0^y (1+y-x)^{(2n-1)/2} \,dx \leq c_1 \rho(y) (1+y)^{(2n+1)/2}, 
\end{multline*}
\begin{multline*}
t^{2(n+1)/3} \int_y^{+\infty} \dots dx \leq c \int_y^{+\infty} e^{-2c_0(x-y)^{3/2}T^{-1/2}} (1+x)^{2\alpha}\, dx \\ =
c \int_0^{+\infty} e^{-2c_0 \theta^{3/2}T^{-1/2}} (1+\theta +y)^{2\alpha}\, d\theta  \\ \leq 
c (1+y)^{2\alpha}  \int_0^{+\infty} e^{-2c_0 \theta^{3/2}T^{-1/2}} (1+\theta)^{2\alpha}\, d\theta \leq c(\alpha) \rho(y).
\end{multline*}
If $y<0$ then
$$
\int \rho(x) |G_x(t,x-y)|^2\, dx = \int_{-\infty}^y\dots dx + \int _y^0 \dots dx +\int_0^{+\infty} \dots dx.
$$
Here
\begin{multline*}
t^{2q(n)} \int_{-\infty}^y \dots dx \leq c \int_{-\infty}^y (1+y-x)^{(2n-1)/2} e^{2\epsilon x}\,dx  \\ \leq
c e^{2\epsilon y} \int_0^{+\infty} (1+\theta)^{1/2} e^{-2\epsilon\theta}\, d\theta \leq c(\epsilon) \rho(y),
\end{multline*}
\begin{multline*}
t^{2(n+1)/3} \int_y^0 \dots dx \leq c \int_y^0 e^{-2c_0(x-y)^{3/2}T^{-1/2}} e^{2\epsilon x}\, dx \\ \leq
c e^{2\epsilon y} \int_0^{+\infty} e^{-2c_0 \theta^{3/2} T^{-1/2}} e^{2\epsilon \theta}\, d\theta \leq c(\epsilon) \rho(y),
\end{multline*}
\begin{multline*}
t^{2(n+1)/3} \int_0^{+\infty} \dots dx \leq c \int _0^{+\infty}e^{-2c_0(x-y)^{3/2}T^{-1/2}} (1+x)^{2\alpha}\, dx \\ \leq
c(T) e^{2\epsilon y} \int_0^{+\infty} e^{-2c_0 \theta^{3/2} T^{-1/2}} (1+\theta)^{2\alpha} e^{2\epsilon\theta} \,d\theta \leq c(\alpha,\epsilon)\rho(y).
\end{multline*}
\end{proof}

\section{}\label{B}

\begin{lemma}\label{LB.1}
Let $\varphi\in W_1^1(I)$ for certain interval $I\subset \mathbb R$, then for a.e. $x\in I$ there exists $|\varphi(x)|'$,
\begin{equation}\label{B.1}
|\varphi|' =
\begin{cases} \displaystyle \frac{\varphi' \bar \varphi + \varphi \bar \varphi'}{2|\varphi|},\quad &\varphi(x)\ne 0,\\ 0,\quad &\varphi(x)=0,
\end{cases}
\end{equation}
in particular,
\begin{equation}\label{B.2}
\bigl||\varphi|'\bigr| \leq |\varphi'|.
\end{equation}
\end{lemma}

\begin{proof}
This formula is obvious if $\varphi(x)\ne 0$. The function $|\varphi(x)|$ is absolutely continuous on every bounded segment in $I$ and, thus, differentiable a.e. on $I$. Let $x_0\in I$ be such point, where $|\varphi(x)|$ is differentiable, $\varphi(x_0)=0$ and there exists a sequence $x_n\to x_0$ such that $\varphi(x_n)=0$. Then
$$
|\varphi(x_0)|' = \lim\limits_{x_n\to x_0} \frac{|\varphi(x_n)| - |\varphi(x_0)|}{x_n-x_0} =0.
$$
It remains to note that the set of isolated zeros of the function $\varphi$ has a zero measure. 
\end{proof}

\begin{lemma}\label{LB.2}
Let $\varphi\in W^2_1(I)$ for certain interval $I\subset \mathbb R$, then for a.e. $x\in I$ there exists $\bigl(|\varphi(x)|' \varphi(x)\bigr)'$ and 
\begin{equation}\label{B.3}
\bigl|\bigl(|\varphi|' \varphi\bigr)'\bigr| \leq 3 |\varphi'|^2 + |\varphi \varphi''|.
\end{equation}
\end{lemma}

\begin{proof}
At the point $x\in I$, where  exist derivatives $\varphi'(x)$, $\varphi''(x)$ and $\varphi(x)\ne 0$, it is easy to see that there exists 
\begin{multline}\label{B.4}
\bigl|\bigl(|\varphi|' \varphi\bigr)'\bigr| = \frac{\varphi'}{2 |\varphi|} (\varphi' \bar \varphi + \varphi \bar \varphi') \\ +
\frac{\varphi}{2 |\varphi|} (\varphi'' \bar\varphi +2|\varphi'|^2 +\varphi\bar\varphi'') - \frac{\varphi}{4 |\varphi|^{3}} (\varphi' \bar\varphi + \varphi \bar\varphi')^2 
 \equiv \Phi(x).
\end{multline}
Set $\Phi(x) =0$ when $\varphi(x)=0$.

For $\delta >0$ consider the functions $g_\delta$ from \eqref{1.7}, then for a.e. $x\in I$
\begin{equation}\label{B.5}
\bigl(g_\delta(|\varphi|^2)\bigr)'  = \frac{(\varphi'\bar\varphi +\varphi \bar\varphi')}{2(|\varphi|^2+\delta)^{1/2}}
\end{equation}
and
$$
\bigl(g_\delta(|\varphi|^2)\bigr)' \varphi \to \frac{\varphi(\varphi'\bar\varphi +\varphi\bar\varphi')}{2|\varphi|}
$$
in $L_1(I)$ when $\delta \to +0$. Moreover, for a.e. $x\in I$ there exists $\bigl(\bigl(g_\delta(|\varphi(x)|^2)\bigr)' \varphi(x)\bigr)'$ and
\begin{multline}\label{B.6}
\bigl|\bigl(\bigl(g_\delta(|\varphi|^2)\bigr)' \varphi\bigr)'\bigr| = \frac{\varphi'}{2(|\varphi|^2+\delta)^{1/2}} (\varphi' \bar \varphi + \varphi \bar \varphi') \\ +
\frac{\varphi}{2(|\varphi|^2+\delta)^{1/2}} (\varphi'' \bar\varphi +2|\varphi'|^2 +\varphi\bar\varphi'') - \frac{\varphi}{4 (|\varphi|^2+\delta)^{3/2}} (\varphi' \bar\varphi + \varphi \bar\varphi')^2.
\end{multline}
Let $\EuScript E$ be s set of points $x\in I$ where either the function $\varphi$ is not twice differentiable or $x$ is an isolated zero of the function $\varphi$. Then this set has a zero measure. Moreover, similarly to the proof of Lemma \ref{LB.1} if $x\not\in \EuScript E$ and $\varphi(x)=0$ then $\varphi'(x)=0$. Then for $x\in I\setminus \EuScript E$
$$
\bigl(\bigl(g_\delta(|\varphi|^2)\bigr)' \varphi\bigr)' \to \Phi(x)
$$
and 
$$
\bigl(\bigl(g_\delta(|\varphi|^2)\bigr)' \varphi\bigr)' \to \Phi
$$
in $L_1(I)$ when $\delta \to +0$. Therefore, there exists $\bigl(|\varphi|' \varphi\bigr)' =\Phi$ and estimate \eqref{B.3} holds.
\end{proof}

\begin{remark}\label{RB.1}
It follows from \eqref{B.2}, \eqref{B.3}, \eqref{B.5} and \eqref{B.6} that uniformly with respect to $\delta\geq 0$
\begin{equation}\label{B.7}
\bigl|\bigl(g_\delta(|\varphi|^2)\bigr)'\bigr| \leq  |\varphi'|,\quad \bigl|\bigl(\bigl(g_\delta(|\varphi|^2)\bigr)' \varphi\bigr)'\bigr| \leq 3|\varphi'|^2 + |\varphi \varphi''|.
\end{equation}
\end{remark}

\begin{lemma}\label{LB.3}
Consider a function of two complex variables $u$ and $v$
\begin{equation}\label{B.8}
P(u,v) \equiv 3|u| v +\frac{u^2 \bar v}{|u|}.
\end{equation}
If $|u|, |v|, |\widetilde u|, |\widetilde v| \leq M$, then
\begin{equation}\label{B.9}
|P(u, v) - P(\widetilde u, \widetilde v)| \leq 6M |u -\widetilde u| + 4M |v -\widetilde v|.
\end{equation}
\end{lemma}

\begin{proof}
In fact,
$$
P(u,v) - P(\widetilde u, \widetilde v) = 3\bigl(|u| - |\widetilde u|\bigr) v + 3|\widetilde u| (v - \widetilde v) + \left(\frac{u^2}{|u|} - \frac{\widetilde u^2}{|\widetilde u|} \right)\bar v + \frac{\widetilde u^2}{|\widetilde u|} \overline{(v- \widetilde v)}.
$$
Let, for example, $|\widetilde u| \leq |u|$, then
$$
\frac{u^2}{|u|} - \frac{\widetilde u^2}{|\widetilde u|} = \frac{u^2 |\widetilde u| - |u| \widetilde u^2}{|u \widetilde u|} = 
\frac{ (u^2 -\widetilde u^2)|\widetilde u| + \widetilde u^2 \bigl(|\widetilde u| - |u|\bigr)}{|u \widetilde u|}
$$
and so
$$
\left| \frac{u^2}{|u|} - \frac{\widetilde u^2}{|\widetilde u|} \right| \leq 
\left(2 +\frac{|\widetilde u|}{|u|}\right) |u - \widetilde u| \leq 3|u - \widetilde u|.
$$
\end{proof}


\begin{thebibliography}{99}

\bibitem{ASV}
M.~Alves, M.~Sep\'ulveda, and O.P.~Vera, {\it Smoothing properties for the higher-order nonlinear Schr\"odinger equation with constant coefficients}, Nonlinear Anal. Theory Methods Appl. {\bf 71} (2009), 948--966.

\bibitem{BOY}
A.~Batal, T.~\"Ozsari, and K.C.~ Yilmaz, {\it Stabilization of higher order linear and nonlinear Schr\"odinger equations on a finite domain: Part I}, Evolution Equ. Control Theory {\bf 10} (2021), 861--919. 

\bibitem{BBV}
E.~ Bisognin, V.~ Bisognin, and O.P.~ Vera, {\it  Stabilization of solutions to higher-order nonlinear Schr\"odinger equation with localized damping}, Electronic J. Differential Equ. {\bf 2007.06} (2007),  1--18.

\bibitem{BV}
V.~ Bisognin and O.P.~Vera, {\it  On the unique continuation property for the higher order nonlinear Schr\"odinger equation with constant coefficients}, Turk. J. Math. {\bf 30} (2006), 1--38.

\bibitem{B} 
J.~Bourgain, {\it  Fourier transform restriction phenomena for certain lattice subsets and applications to non-linear evolution equations,
part II: the KdV equation}, Geom. Funct. Anal. {\bf 3} (1993), 209--262.

\bibitem{Car}
X. Carvajal, {\it Sharp global well-posedness for a higher order Schr\"odinger equation}, J. Fourier Anal. Appl. {\bf 303} (2006), 53--70.

\bibitem{CL}
X.~ Carvajal and F.~ Linares, {\it A higher order nonlinear Schr\"odinger equation with variable coefficients}, Differential Integral Equ. {\bf 16} (2003), 1111--1130.

\bibitem {CN}
X. Carvajal and W. Neves, {\it Persistence of solutions to higher order nonlinear Schr\"odinger equation}, J. Differential Equ. {\bf 249} (2010), 2214--2236.

\bibitem{CP}
X. Carvajal and M. Panthee, {\it Unique continuation for a higher order nonlinear Schr\"odinger equation}, J. Math. Anal. Appl. {\bf 303} (2005), 188--207.

\bibitem{CCFSV}
M.M.~Cavalcanti, W.C.~Corr\^ ea, A.V.~Faminskii, M.A.~Sepulv\'eda C., and R.~V\'ejar-Asem, {\it Well-posedness and asymptotic behavior of a generalized higher order nonlinear Schr\"{o}dinger equation with localized dissipation}, Comp. Math. Appl. {\bf 96} (2021), 188--208.

\bibitem{CDFN}
M.M.~Cavalcanti, V.N.~ Domingos Cavalcanti, A.~Faminskii, and F.~Natali, {\it Decay of solutions to damped Korteweg--de~Vries type equation}, Appl. Math. Optim. {\bf 65} (2012), 221--251.

\bibitem{Chen} 
M.~Chen, {\it Stabilization of the higher order nonlinear Schr\"odinger equation with constant coefficients}, Proc. Indian Acad. Sci. (Math. Sci.) {\bf 128} (2018), Art.~39. 

\bibitem{CCPV}
V.~Ceballos, J.~Carlos, F.~Pavez, and O.P.~Vera, {\it Exact boundary controllability for higher order nonlinear Schr\"odinger equations with constant coefficients}. Electronic J. Differential Equ. {\bf 2005.122} (2005), 1--31.

\bibitem{CT}
S.~Cui and T.~Tao, {\it Strichartz estimates for dispersive equations and solvability of the Kawahara equation}, 
J. Math. Anal. Appl. \textbf{304} (2005), 683--702.

\bibitem{Fib}
G.~Fibich, {\it Adiabatic law for self-focusing of optical beams}, Opt. Lett. {\bf 21} (1996), 1735--1737.

\bibitem{Guo}
Z.~Guo, {\it Global well-posedness of Korteweg--de~Vries equation in $H^{-3/4}(\mathbb R)$}, J. Math. Pures Appl. {\bf 91} (2009), 583--597.

\bibitem{HK}
A.~Hasegawa and Y.~ Kodama, {\it  Nonlinear pulse propagation in a monomode dielectric guide}, IEEE J. Quantum Electron {\bf 23} (1987), 510--524.

\bibitem{Kod} 
Y.~Kodama, {\it Optical solitons in a monomode fiber}, J. Stat. Phys. {\bf 39} (1985), 597--614.

\bibitem{KF} 
S.N.~Kruzhkov and A.V.~ Faminski. {\it  Generalized solutions of the Cauchy problem for the Korteweg~--de~Vries equation}, Mat. Sb. {\bf 120(162)} (1983), 396--425; English transl. in Sb. Math. {\bf 48} (1984), 391--421.

\bibitem{KC}
H.~ Kumar and  F.~ Chand, {\it Dark and bright solitary waves solutions of the higher order nonlinear Schr\"odinger equation with self-steeping and self-frequency shift effects}, J. Nonlinear Opt. Phys. Mater. {\bf 22} (2013), Art.~1350001.

\bibitem{Laurey} 
C.~Laurey, {\it The Cauchy problem for a third order nonlinear Schr\"odinger equation}, Nonlinear Anal. Theory Methods Appl. {\bf 29} (1997), 121--158.

\bibitem {LP}
F.~Linares and G.~Ponce, {\it Introduction to nonlinear dispersive equations}, Universitext, Springer, 2009.

\bibitem{Sh} 
N.A.~Shananin, \textit{Partial quasianalyticity of distribution solutions of weakly nonlinear differential equations with weights assigned to derivatives}, Matemat. Zametki \textbf{68} (2000), 608--619; English transl. in Math. Notes \textbf{68} (2000), 519--527.

\bibitem{Staf}
G.~ Staffilani, {\it On the generalized Korteweg--de~Vries-type equations}, Differential Integral Equ. {\bf10} (1997), 777--796.

\end{thebibliography}
\end{document}